\documentclass[11pt,times,twoside]{article}

\usepackage[leqno]{amsmath}
\usepackage{graphicx}
\usepackage{latexsym}
\usepackage{amsmath,amsfonts,amssymb,amsthm,euscript,makeidx,color,mathrsfs}
\usepackage{enumerate}

\usepackage[body={14.5cm,22cm}, left=2.5cm, right=2.5cm, top=3cm]{geometry}
\usepackage[colorlinks,linkcolor=blue,anchorcolor=green,citecolor=red]{hyperref}

\def\sqr#1#2{{\vcenter{\vbox{\hrule height.#2pt
              \hbox{\vrule width.#2pt height#1pt \kern#1pt \vrule width.#2pt}
              \hrule height.#2pt}}}}
\def\5n{\negthinspace \negthinspace \negthinspace \negthinspace \negthinspace }
\def\4n{\negthinspace \negthinspace \negthinspace \negthinspace }
\def\3n{\negthinspace \negthinspace \negthinspace }
\def\2n{\negthinspace \negthinspace }
\def\1n{\negthinspace }

\def\dbE{\mathbb{E}}

\def\dbP{\mathbb{P}}

\newcommand{\ep}{\varepsilon}

\newcommand{\R}{\mathbb{R}}

\newcommand{\E}{\mathbb{E}}
\newcommand{\PP}{\mathbb{P}}

\def\ds{\displaystyle}

\def\ns{\noalign{\ss}}
\def\cU{{\cal U}}

\def\ba{\bar{a}}

\def\bc{\bar{c}}

\def\be{\bar{e}}

\def\bl{\bar{l}}

\def\bp{\bar{p}}

\def\br{\bar{r}}

\def\bt{\bar{t}}

\def\ss{\smallskip}
\def\ms{\medskip}

\def\qq{\qquad}

\def\rf{\eqref}

\def\cd{\cdot}

\def\({\Big (}
\def\){\Big )}
\def\[{\Big[}
\def\]{\Big]}

\def\bde{\begin{definition}\label}
\def\ede{\end{definition}}
\def\be{\begin{equation}}
\def\bel{\begin{equation}\label}
\def\ee{\end{equation}}
\def\bt{\begin{theorem}\label}
\def\et{\end{theorem}}
\def\bc{\begin{corollary}\label}
\def\ec{\end{corollary}}
\def\bl{\begin{lemma}\label}
\def\el{\end{lemma}}
\def\bp{\begin{proposition}\label}
\def\ep{\end{proposition}}
\def\bas{\begin{assumption}\label}
\def\eas{\end{assumption}}
\def\br{\begin{remark}\label}
\def\er{\end{remark}}
\def\bex{\begin{example}\label}
\def\ex{\end{example}}
\def\ba{\begin{array}}
\def\ea{\end{array}}
\def\ed{\end{document}}

\def\olb\
\def\eps{\epsilon}
\def\square#1{\vbox{\hrule\hbox{\vrule height#1%
     \kern#1\vrule}\hrule}}
\def\rectangle#1#2{\vbox{\hrule\hbox{\vrule height#1%
     \kern#2\vrule}\hrule}}

\font\tenbb=msbm10 \font\sevenbb=msbm7 \font\fivebb=msbm5

\newfam\bbfam
\scriptscriptfont\bbfam=\fivebb \textfont\bbfam=\tenbb
\scriptfont\bbfam=\sevenbb

\newtheorem{theorem}{\hskip 1.3em Theorem}[section]
\newtheorem{definition}[theorem]{\hskip 1.3em Definition}
\newtheorem{proposition}[theorem]{\hskip 1.3em Proposition}
\newtheorem{corollary}[theorem]{\hskip 1.3em Corollary}
\newtheorem{lemma}[theorem]{\hskip 1.3em Lemma}
\newtheorem{remark}[theorem]{\hskip 1.3em Remark}
\newtheorem{example}[theorem]{\hskip 1.3em Example}

\newtheorem{assumption}[theorem]{\hskip 1.3em Assumption}

\makeatletter
   
   \@addtoreset{equation}{section}
\makeatother

\begin{document}

\title{\Large \bf Mean-field anticipated BSDEs driven by fractional Brownian motion and related stochastic control problem}
\author
{\textbf{Soukaina Douissi}$^{1}$,
	\textbf{Jiaqiang Wen}$^{2}$
	, \textbf{Yufeng Shi}$^{2,3}$
	\vspace{3mm} \\
	\normalsize{$^{1}$ Laboratory LIBMA, Faculty Semlalia, University Cadi Ayyad, Marrakech, Morocco}\\
	\normalsize{$^{2}$ School of Mathematics, Shandong University, Jinan 250100, China}\\
	\normalsize{$^{3}$School of Statistics, Shandong University of Finance and Economics, Jinan 250014, China}
}
\date{}

\renewcommand{\thefootnote}{\fnsymbol{footnote}}

\footnotetext[0]{The first author is supported by the Erasmus + International Credit mobility between Linnaeus University, V\"{a}xj\"{o}, Sweden and Cadi Ayyad University, Marrakech, Morocco for the academic year 2016-2017.
\\	
The second and third authors are supported  by NNSF of China (Grant Nos. 11371226, 11071145, 11526205, 11626247 and 11231005), the Foundation for Innovative Research Groups of National Natural Science Foundation of China (Grant No. 11221061) and the 111 Project (Grant No. B12023).\\
{E-mail addresses: douissi.soukaina@gmail.com, jqwen59@gmail.com, yfshi@sdu.edu.cn}
}

\maketitle
\begin{abstract}
In this paper, we focus on mean-field anticipated backward stochastic differential equations (MF-BSDEs, for short) driven by fractional Brownian motion with Hurst parameter $H>1/2$.
First, the existence and uniqueness of this new type of BSDEs are established using two different approaches.
Then, a comparison theorem for such BSDEs is obtained.
Finally, as an application of this type of equations, a related stochastic optimal control problem is studied.
\end{abstract}

\textbf{Keywords}: Mean-field backward stochastic differential equation; Anticipated backward stochastic differential equation;  Fractional Brownian motion; Stochastic control.

\textbf{2010 Mathematics Subject Classification}: 60H10, 60H20, 60G22, 93E20.

\section{Introduction}

A centered Gaussian process $B^{H}=\{ B^{H}_{t},t\geq 0 \}$ is called a
fractional Brownian motion (fBm, for short) with Hurst parameter $H \in (0, 1)$ if its covariance is
\begin{equation*}
 \dbE(B^{H}_{t}B^{H}_{s}) = \frac{1}{2} (t^{2H} + s^{2H} - |t-s|^{2H}), \qq t,s \geq 0.
\end{equation*}
When $H=1/2$, this process becomes a classical Brownian motion.
For $H>1/2$, $B^{H}$ exhibits the property of long range dependence,
which makes the fBm an important driving noise in many fields such as finance, telecommunication networks, and physics.

\ms

In 1990, the nonlinear backward stochastic differential equations (BSDEs, for short) were introduced by
 Pardoux and Peng \cite{Peng}.
In the next two decades, BSDEs have been widely used in different fields of mathematical finance (see \cite{Peng2}), stochastic control (see \cite{Yong5}), and partial differential equations (see \cite{Peng92}).
At the same time, for better applications, BSDE itself has been developed into many different branches. For example, Buckdahn et al. \cite{Buckdahn} and Buckdahn, Li and Peng \cite{Buckdahn2} introduced the so-called mean-field BSDEs, owing to the fact that mathematical mean-field approaches have important applications in many domains,
such as  Economics, Physics and Game Theory (see Lasry and Lions \cite{Lasry}, Buckdahn et al. \cite{Buckdahn2017} and the papers therein). Peng and Yang \cite{PY2009} introduced a new type of BSDEs, called anticipated BSDEs, which can be regarded as a new duality type of stochastic differential delay equations.
Furthermore, BSDEs driven by fractional Brownian motion, also known as fractional BSDEs,
with Hurst parameter $H>1/2$ were studied by Hu and Peng \cite{Hu}.
Then Maticiuc and Nie \cite{Maticiuc} obtained some general results of fractional BSDEs through a rigorous approach. Buckdahn and Jing \cite{Buckdahn3} studied fractional mean-field stochastic differential equations (SDEs, for short) with $H>1/2$ and a stochastic control problem.
Some other recent developments of fractional BSDEs can be found in
Bender \cite{Bender}, Borkowska \cite{Borkowska}, Maticiuc and Nie \cite{Maticiuc},
 Wen and Shi \cite{Wen,WS}, etc., among theory and applications.

\ms

As another important development of BSDEs, mean-field anticipated BSDEs (MF-ABSDEs, for short) driven by fBm have significant applications in stochastic optimal control problems with delay. In \cite{DHA},  Agram, Douissi and Hilbert solved the optimal control problem of mean-field stochastic delayed differential equations, where they considered the integral with respect to the fBm  of the adjoint BSDE in the Wick sense, (see \cite{BOZ}), they proved the set of necessary and sufficient maximum principles and gave some applications.
In our work, we investigate another approach to solve this problem. Namely, we focus on MF-ABSDEs driven by fBm when the integral with respect to the fBm is in the divergence sense, (see Decreusefond and \"{U}st\"{u}nel \cite{Decreusefond}, and  Nualart \cite{Nualart}). Specifically, we study the following equation,
\bel{0}\left\{\ba{ll}
\ds Y_{t}=g(\eta_{T})+\int_t^T\dbE'[f(s,\eta_{s},Y'_{s},Z'_{s},Y_{s},Z_{s},
Y'_{s+\delta(s)},Z'_{s+\zeta(s)},Y_{s+\delta(s)},Z_{s+\zeta(s)})]ds\\
\ns\ds\qq\qq \ \ \ -\int_t^TZ_{s}dB_{s}^{H}, \qq t\in [0,T];\\
\ns\ds Y_{t}=g(\eta_{t}), \qq  Z_{t}=h(\eta_{t}), \qq t\in[T,T+K],
\ea\right.\ee
where $\delta(\cdot)$ and $\zeta(\cdot)$ are two deterministic $\mathbb{R}^{+}$-valued continuous functions defined on $[0,T]$. First, we use two different approaches to prove the existence and uniqueness of solutions of MF-ABSDE \rf{0}.
Interestingly, the conditions required by the first approach are weaker then the second one, however, the second approach is more convenient than the first one.
Second, as a fundamental tool, the comparison theorem plays an important role in the theory and applications of BSDEs. We establish a comparison theorem for this type of MF-ABSDEs.
Finally, as an application of such BSDEs, a stochastic optimal control problem is studied and the related sufficient maximum principle is obtained.

\ms

We organize this article as follows.
Some  preliminaries about fBm and other required definitions are presented in Section 2.
The existence and uniqueness of fractional MF-ABSDEs are proved by two different approaches in Section 3.
We derive a comparison theorem for such type of equations in Section 4
and investigate a stochastic optimal control problem in Section 5.

\section{Preliminaries}

We recall, in this section, some basic results of fractional Brownian motion and the differentiability of functions of measures.

\subsection{Fractional Brownian motion}

In this subsection, some preliminaries about fractional Brownian motion are presented.
For a deeper discussion, the readers may refer to the articles such as Decreusefond and \"{U}st\"{u}nel \cite{Decreusefond}, Hu \cite{Hu3} and  Nualart \cite{Nualart}, etc.

\ms

Assume $B^{H}=\{ B^{H}_{t},t\geq 0 \}$ is a fBm defined on a complete probability space $(\Omega,\mathcal{F},\dbP)$,
and the filtration $\mathcal{F}$ is generated by $B^{H}$.
Let $H >1/2$ throughout this paper.
Moreover, we denote $\phi(x) = H(2H - 1)|x|^{2H-2},$ where $x \in \mathbb{R}$,
and suppose $\xi$ and $\psi$ are two continuous functions defined in $[0,T]$.
Define
\begin{equation}\label{11}
  \langle \xi,\psi \rangle_{T} = \int_0^T \int_0^T \phi(u-v) \xi_{u} \psi_{v} dudv, \ \
  and \ \   \| \xi \|_{T}^{2} =  \langle \xi,\xi  \rangle_{T}.
\end{equation}
Then $\langle \xi,\psi \rangle_{T}$ is a Hilbert scalar product.
Under this scalar product, we denote by $\mathcal{H}$ the completion of the continuous functions.
Besides, denote by $\mathcal{P}_{T}$ the set of all polynomials of fBm in $[0,T]$, i.e.,
every element of $\mathcal{P}_{T}$ is of the form
\begin{equation*}
  \Phi(\omega) = h \left(\int_0^T \xi_{1}(t) dB_{t}^{H},...,\int_0^T \xi_{n}(t) dB_{t}^{H} \right),
\end{equation*}
where $h$ is a polynomial function and $\xi_{i}\in\mathcal{H}, i=1,2,...,n$.
In addition,
Malliavin derivative operator $D_{s}^{H}$ of $\Phi\in \mathcal{P}_{T}$ is defined by
\begin{equation*}
  D_{s}^{H}\Phi = \sum\limits_{i=1}^{n} \frac{\partial h}{\partial x_{i}}
               \left(\int_0^T \xi_{1}(t) dB_{t}^{H},...,\int_0^T \xi_{n}(t) dB_{t}^{H} \right)\xi_{i}(s), \qq
               s\in [0,T].
\end{equation*}
Since the derivative operator
  $D^{H}:L^{2}(\Omega,\mathcal{F}, P) \rightarrow (\Omega,\mathcal{F}, \mathcal{H})$ is closable,
one can denote by $\mathbb{D}^{1,2}$ the completion of $\mathcal{P}_{T}$ under the following norm
$$\| \Phi \|^{2}_{1,2} \triangleq \dbE|\Phi|^{2} + \dbE\|D^{H}_{s} \Phi\|^{2}_{T}.\qq\qq\qq \ $$
Furthermore, we introduce the following derivative
\begin{equation*}
  \mathbb{D}_{t}^{H}\Phi = \int_0^T \phi(t-s) D_{s}^{H}\Phi ds, \qq t\in[0,T]. \
\end{equation*}
Now, let us consider the adjoint operator of Malliavin derivative operator $D^{H}$.
We call this operator the divergence operator, which represents the divergence type integral and
 is denoted by $\delta(\cdot)$.

\begin{definition}
A process $u\in L^{2}(\Omega\times[0,T];\mathcal{H})$ is said to belongs to the domain $Dom(\delta)$,
if there exists $\delta(u)\in L^{2}(\Omega,\mathcal{F},\dbP)$ satisfying the following duality relationship
\begin{equation*}\label{}
  \dbE(\Phi\delta(u))=\dbE(\langle D^{H}_{\cdot} \Phi,u \rangle_{T}), \ \ for \ \ every \ \ \Phi\in\mathcal{P}_{T}.
\end{equation*}
Moreover, if $u\in Dom(\delta)$, the divergence type integral of $u$ w.r.t. $B^{H}$ is defined by putting
$\int_0^T u_{s} d B^{H}_{s}=: \delta(u)$.
\end{definition}

It should be pointed out that, in this paper, unless otherwise specified,
 the $d B^{H}$-integral represents the divergence type integral.

\begin{proposition}[Hu \cite{Hu3}, Proposition 6.25]\label{2}
Let $\mathbb{L}^{1,2}_{H}$ be the space of all processes $F : \Omega\times[0,T] \rightarrow \mathcal{H}$ satisfying
$ \dbE \left( \| F \|_{T}^{2} + \int_0^T \int_0^T |\mathbb{D}_{s}^{H}F_{t}|^{2} dsdt \right) < \infty.$
Then, if $F \in \mathbb{L}^{1,2}_{H}$, the divergence type integral
$\int_0^T F_{s} dB_{s}^{H}$ exists in $L^{2}(\Omega,\mathcal{F}, \dbP)$, and
\begin{equation*}
  \dbE \left( \int_0^T F_{s} dB_{s}^{H} \right) = 0; \qq
   \dbE \left( \int_0^T F_{s} dB_{s}^{H} \right)^{2}
 =\dbE \left( \| F \|_{T}^{2} + \int_0^T \int_0^T \mathbb{D}_{s}^{H}F_{t} \mathbb{D}_{t}^{H}F_{s} dsdt \right).
\end{equation*}
\end{proposition}

\begin{proposition}[Hu \cite{Hu3}, Theorem 10.3]\label{4}
Suppose  $g$ and  $f$ are two deterministic continuous functions. Let
\begin{equation*}
  X_{t} = X_{0} + \int_0^t g_{s} ds + \int_0^t f_{s} dB_{s}^{H}, \qq t\in [0,T],\qq\qq \
\end{equation*}
where $X_{0}$ is a constant. Then, if $F \in C^{1,2}([0, T ] \times \mathbb{R})$, one has
$$\ba{ll}
\ds F(t,X_{t})=F(0,X_{0})+ \int_0^t \frac{\partial F}{\partial s}(s,X_{s}) ds
  + \int_0^t \frac{\partial F}{\partial x}(s,X_{s})g_{s} ds \\
\ns\ds\qq\qq \ \ + \int_0^t \frac{\partial F}{\partial x}(s,X_{s}) f_{s} dB_{s}^{H}
+ \frac{1}{2}\int_0^t \frac{\partial^{2} F}{\partial x^{2}}(s,X_{s}) \bigg[\frac{d}{ds} \| f \|_{s}^{2}\bigg] ds,
  \qq  t\in [0,T].
\ea$$
\end{proposition}

\begin{proposition}[Hu \cite{Hu3}, Theorem 11.1]\label{3}
 For $i = 1, 2$, let $g_{i}$ and $f_{i}$ be two real valued processes satisfying
 $\dbE \int_0^T (|g_{i}(s)|^{2} + |f_{i}(s)|^{2}) ds < \infty$.
Moreover, assume that $D^{H}_{t}f_{i}(s)$ is continuously differentiable in its arguments
 $(s,t)\in [0,T]^{2}$ for almost every $\omega \in \Omega$, and
$\dbE \int_0^T \int_0^T |\mathbb{D}_{t}^{H}f_{i}(s)|^{2} dsdt < \infty$.
Denote
\begin{equation*}
  X_{i}(t) = \int_0^t g_{i}(s) ds + \int_0^t f_{i}(s) dB_{s}^{H}, \qq t\in [0,T]. \qq\ \ \ \ \
\end{equation*}
Then
\begin{equation*}
\begin{split}
   X_{1}(t)X_{2}(t) =& \int_0^t X_{1}(s)g_{2}(s) ds + \int_0^t X_{1}(s)f_{2}(s) dB_{s}^{H}
                      +\int_0^t X_{2}(s)g_{1}(s) ds \\
                    &+ \int_0^t X_{2}(s)f_{1}(s) dB_{s}^{H} +\int_0^t \mathbb{D}_{s}^{H}X_{1}(s)f_{2}(s) ds + \int_0^t \mathbb{D}_{s}^{H}X_{2}(s)f_{1}(s) ds.
\end{split}
\end{equation*}
\end{proposition}

\begin{proposition}[Wen and Shi \cite{Wen}, Lemma 3.1]\label{WS}
Suppose $g$ is a given differentiable function with polynomial growth and $f$ is a $C^{0,1}_{pol}$-continuous function.
Then BSDE
\begin{equation*}
    Y_{t}=g(\eta_{T}) + \int_t^T f(s,\eta_{s}) ds - \int_t^T Z_{s} dB_{s}^{H}\qq\qq\qq \
\end{equation*}
admits a unique solution
 $(Y_{\cdot},Z_{\cdot}) \in \widetilde{\mathcal{V}}_{[0,T]} \times \widetilde{\mathcal{V}}^{H}_{[0,T]}$ (see \rf{3.99} for the definition of these spaces).
Moreover, the following estimate holds,
\begin{equation}\label{34}
\begin{split}
    & {\dbE}\left(e^{\beta t}|Y_{t}|^{2} + \frac{\beta}{2}\int_t^T e^{\beta s}|Y_{s}|^{2} ds + \frac{2}{M}\int_t^T s^{2H-1}e^{\beta s}|Z_{s}|^{2} ds\right)\\
\leq& {\dbE}\left(e^{\beta T}|g(\eta_{T})|^{2} + \frac{2}{\beta}\int_t^T e^{\beta s}|f(s,\eta_{s})|^{2} ds \right).
\end{split}
\end{equation}
where $M > 0$ is a suitable constant and $\beta > 0$.
\end{proposition}

\subsection{Differentiability of Functions of Measures}

We recall now some definitions related to the differentiability with respect to functions of measures that we will need in Section 5. Let $\mathcal{P}(\R)$ be the space of all probability measures on $(\R, \mathcal{B}(\R))$. We denote by $\mathcal{P}_p(\R)$ the subspace of  $\mathcal{P}(\R)$ of order $p$, which means that $\mathcal{P}_p(\R)\triangleq\{m\in\mathcal{P}(\R):\int_{\R} |x|^p m(dx)<+\infty\}$. The notion of differentiability for functions of measures that we will use in the paper is inspired from the notes of Cardaliaguet \cite{Cardaliaguet} and the work of Carmona and Delarue \cite{carmona2}. It's based on the \textit{lifting} of functions $m \in \mathcal{P}_2(\R) \mapsto \sigma(m)$ into functions ${\xi}^{\prime} \in L^{2}({\Omega}; \mathbb{R}) \mapsto {\sigma}^{\prime}({\xi}^{\prime})$, over some probability space $({\Omega},\mathcal{F},\mathbb{P})$, by setting ${\sigma}^{\prime}
({\xi}^{\prime}) \triangleq \sigma({\mathbb{P}}_{{\xi}^{\prime}})$.
\begin{definition}
	A function $\sigma$ is said to be differentiable at $ m_{0} \in \mathcal{P}_2(\R)$, if there exists a random variable ${\xi}^{\prime}_{0}\in L^2({\Omega},\mathcal{F},{\PP})$ over some probability space $({\Omega},\mathcal{F},{\PP})$ with ${\PP}_{{\xi}^{\prime}_{0}}= m_{0}$ such that ${\sigma}^{\prime}: L^2({\Omega},\mathcal{F},{\PP}) \to \R$ is Fr\'echet differentiable at ${\xi}^{\prime}_{0}$.
\end{definition}

We suppose for simplicity that ${\sigma}^{\prime}: L^2({\Omega},\mathcal{F},\PP) \to \R$ is Fr\'echet differentiable. We denote its Fr\'echet derivative at ${\xi}^{\prime}_{0}$ by $D{\sigma}^{\prime}({\xi}^{\prime}_{0})$. Recall that $D{\sigma}^{\prime}({\xi}^{\prime}_{0}):L^2({\Omega},\mathcal{F}, \PP) \to \R$ is a continuous linear mapping; i.e. $D{\sigma}^{\prime}({\xi}^{\prime}_{0})\in L( L^2({\Omega},\mathcal{F}, \PP),\R)$.
With the identification that $L(L^2({\Omega},\mathcal{F}, \PP),\R) \equiv L^2({\Omega},\mathcal{F},\PP)$ given by Riesz representation theorem, $D{\sigma}^{\prime}({\xi}^{\prime}_{0})$ is viewed as an element of
$L^2({\Omega},\mathcal{F}, \PP)$, hence we can write
$$
\sigma(m)-\sigma(m_{0})= {\sigma}^{\prime}({\xi}^{\prime})-{\sigma}^{\prime}({\xi}^{\prime}_{0})=
{\mathbb{E}}[(D{\sigma}^{\prime})({{\xi}^{\prime}}_{0})\cd({\xi}^{\prime}-{\xi}^{\prime}_{0})] +o({\mathbb{E}}[|{\xi}^{\prime}-{\xi}^{\prime}_{0}|^{2}]^{1/2}),\  \textrm{as}\  {\mathbb{E}}[|{\xi}^{\prime}-{\xi}^{\prime}_{0}|^{2}]^{1/2} \to 0.
$$
where ${\xi}^{\prime}$ is a random variable with law $m$. Moreover, according to Cardaliaguet \cite{Cardaliaguet}, there exists a Borel function $h_{m_{0}}:\R\to\R$, such that $D{\sigma}^{\prime}({\xi}^{\prime}_{0}) =h_{m_{0}}({\xi}^{\prime}_{0})$, ${\PP}$-a.s. We define the derivative of $\sigma$ with respect to the measure at $m_{0}$ by putting $\partial_m\sigma(m_{0})(x) : =h_{m_{0}}(x)$. Notice that $\partial_m\sigma(m_{0})(x)$ is defined $m_{0}(d x)$-a.e. uniquely. Therefore, the following differentiation formula is invariant by modification of the space ${\Omega}$ where the random variables ${\xi}^{\prime}_{0}$ and ${\xi}^{\prime}$ are defined, i.e.
$$
\sigma(m)-\sigma(m_{0})= {\mathbb{E}}[
\partial_m\sigma(m_{0})({\xi}^{\prime}_{0})\cd({\xi}^{\prime}- {\xi}^{\prime}_{0})]+o({\mathbb{E}}[|{\xi}^{\prime}- {\xi}^{\prime}_{0}|^{2}]^{1/2}),\  \textrm{as}\ {\mathbb{E}}[|{\xi}^{\prime}- {\xi}^{\prime}_{0}|^{2}]^{1/2} \to 0.
$$
whenever ${\xi}^{\prime}$ and ${\xi}^{\prime}_{0}$ are random variables with laws $m$ and $m_{0}$ respectively. \\
\\
\textbf{ Joint concavity:} We will need the joint concavity of a function on $(\R\times \mathcal{P}_2(\R))$. A differentiable function $b$ defined on  $(\R \times \mathcal{P}_2(\R))$ is concave, if for every   $(x^\prime,m^\prime)$ and $(x,m) \in  (\R \times \mathcal{P}_2(\R))$, we have
\begin{equation*}
\label{concavity}
\begin{aligned}
b(x^{\prime},m^{\prime})-b(x,m)- \partial_x b(x,m) (x^\prime-x) -{{\E}}[\partial_m b(x,m)({X})({X}^\prime- {X})] \leq 0,
\end{aligned}
\end{equation*}
whenever ${X}, {X}^\prime\in L^2({\Omega},\mathcal{F},{\PP};\R)$ with laws $m$ and $m^{\prime}$ respectively.
\section{Well-posedness}

The existence and uniqueness of mean-field anticipated BSDEs driven by fBm are proved here by using two different
approaches. For simplify the presentation, we only discuss the one dimensional case in this paper. Let
\begin{equation*}
\eta_{t} = \eta_{0} + \int_0^t b_{s} ds + \int_0^t \sigma_{s} dB_{s}^{H},
\end{equation*}
where $\eta_{0}$ is a constant, and $b$ and $\sigma$ are two deterministic differentiable functions such that $\sigma_{t} \neq 0$ (then either $\sigma_{t} < 0$ or $\sigma_{t} > 0$),  $t\in [0,T]$.
We recall that (see (\ref{11}))
\begin{equation*}
  \| \sigma \|_{t}^{2} = H(2H-1) \int_0^t \int_0^t |u-v|^{2H-2} \sigma_{u} \sigma_{v} dudv.
\end{equation*}
So $\frac{d}{dt}(\| \sigma \|_{t}^{2})= 2 \hat{\sigma}_{t} \sigma_{t}>0$ for $t\in (0,T]$, where
$\hat{\sigma}_{t} = \int_0^t \phi(t-v) \sigma_{v} dv$.

\ms

Now, we denote the (non-completed) product space of $(\Omega,\mathcal{F},\dbP)$ by $(\bar{\Omega},\bar{\mathcal{F}}, \bar{\dbP})$ $ = (\Omega\times \Omega,\mathcal{F}\otimes \mathcal{F}, \dbP\otimes \dbP)$,
and denote the filtration of this product space by $\bar{\mathbb{F}} = \{ \bar{\mathcal{F}}_{t} = \mathcal{F} \otimes \mathcal{F}_{t}, 0\leq t\leq T \}$.
A random variable, originally defined on $\Omega$,
$\xi\in L^{0}(\Omega,\mathcal{F}, \dbP;\mathbb{R})$ is canonically extended to $\bar{\Omega}$:
$\xi'(\omega',\omega)=\xi(\omega'), \ (\omega',\omega)\in \bar{\Omega} = \Omega\times \Omega$.
On the other hand, for every $\theta\in L^{1}(\bar{\Omega},\bar{\mathcal{F}}, \bar{\dbP})$, the random variable
$\theta(\cdot,\omega):\Omega\rightarrow \mathbb{R}$ is in $L^{1}(\Omega,\mathcal{F}, \dbP), \  \dbP(d \omega), \ a.s.$,
and its expectation is denoted by
\begin{equation*}
  \dbE'[\theta(\cdot,\omega)] = \int_{\Omega} \theta(\omega',\omega) \dbP(d \omega').
\end{equation*}
Then we have $\dbE'[\theta]=\dbE'[\theta(\cdot,\omega)]\in L^{1}(\Omega,\mathcal{F}, \dbP)$. In addition,
\begin{equation*}
 \bar{\dbE}[\theta]\bigg(= \int_{\bar{\Omega}} \theta d \bar{\dbP}
 = \int_{\Omega}  \dbE'[\theta(\cdot,\omega)] \dbP(d \omega) \bigg) = \dbE\big[\dbE'[\theta]\big].
\end{equation*}

In the following, we investigate the existence and uniqueness of BSDE \rf{0}. And for simplicity of presentation, we rewrite BSDE \rf{0} into a differential form,
\begin{equation}\label{31}
  \begin{cases}
    -dY_{t}= \dbE'[f(t,\eta_{t},Y'_{t},Z'_{t},Y_{t},Z_{t},Y'_{t+\delta(t)},Z'_{t+\zeta(t)},Y_{t+\delta(t)},Z_{t+\zeta(t)})] dt
             - Z_{t} dB_{t}^{H}, \ \ \ t\in [0,T]; \\
      Y_{t}=g(\eta_{t}), \ \  Z_{t}=h(\eta_{t}), \ \ \ t\in[T,T+K],
  \end{cases}
\end{equation}
where $K\geq 0$ is a constant, $\delta(\cdot)$ and $\zeta(\cdot)$ are two deterministic $\mathbb{R}^{+}$-valued continuous functions defined on $[0,T]$ satisfying the following two issues:
\begin{itemize}
  \item [(i)] For all $t\in[0,T]$,
    \begin{equation*}
      t + \delta(t) \leq T+K, \qq t + \zeta(t) \leq T+K.
    \end{equation*}
  \item [(ii)] There exists a constant $L \geq 0$ such that  for all nonnegative and integrable $m(\cdot)$,
   \begin{equation*}
     \int_t^T m(s + \delta(s)) ds \leq L \int_t^{T+K} m(s) ds, \ \ \
     \int_t^T m(s + \zeta(s)) ds \leq L \int_t^{T+K} m(s) ds,  \ \ \ t \in [0,T].
   \end{equation*}
\end{itemize}

\begin{remark} \label{Rem}
Owing to our notation, we mark that the coefficient of Eq. (\ref{31}) is explained by:
\begin{equation*}
\begin{split}
     &\dbE'[f(t,\eta_{t},Y'_{t},Z'_{t},Y_{t},Z_{t},Y'_{t+\delta(t)},Z'_{t+\zeta(t)},Y_{t+\delta(t)},Z_{t+\zeta(t)})](\omega)\\
   =& \dbE'[f(t,\eta_{t}(\omega),Y'_{t},Z'_{t},Y_{t}(\omega),Z_{t}(\omega),
     Y'_{t+\delta(t)},Z'_{t+\zeta(t)},Y_{t+\delta(t)}(\omega),Z_{t+\zeta(t)}(\omega))]\\
   =& \int_{\Omega} f(t,\eta_{t}(\omega),Y_{t}(\omega'),Z_{t}(\omega'),Y_{t}(\omega),Z_{t}(\omega),
     Y_{t+\delta(t)}(\omega'),Z_{t+\zeta(t)}(\omega'),Y_{t+\delta(t)}(\omega),Z_{t+\zeta(t)}(\omega)) \dbP (d\omega').
\end{split}
\end{equation*}
\end{remark}
From the above remark, combining the definition of expectation, we have the following two special cases:
\bel{34.2}\ba{ll}
\ds \dbE'[f(t,Y'_{t},Z'_{t},Y'_{t+\delta(t)},Z'_{t+\zeta(t)})]
 =\dbE[f(t,Y_{t},Z_{t},Y_{t+\delta(t)},Z_{t+\zeta(t)})], \\
\ns\ds \dbE'[f(t,\eta_{t},Y_{t},Z_{t},Y_{t+\delta(t)},Z_{t+\zeta(t)})]
 =f(t,\eta_{t},Y_{t},Z_{t},Y_{t+\delta(t)},Z_{t+\zeta(t)}).
\ea\ee
Before giving the definition of  solutions of BSDE (\ref{31}),
we introduce the following sets,

\begin{itemize}
  \item [$\bullet$] $L^{2}(\mathcal{F}_{r};\mathbb{R}) =\Big\{\xi:\Omega\rightarrow \mathbb{R} \big| \xi$
        is $\mathcal{F}_{r}$-measurable, $\dbE[|\xi|^{2}]< \infty\Big\}$;
  \item [$\bullet$]
   $C_{pol}^{1,3}([0,T]\times \mathbb{R})=\Big\{ \varphi\in C^{1,3}([0, T] \times \mathbb{R}),$  and
     all derivatives  of   $\varphi$ are of polynomial  growth$\Big\}$;
  \item [$\bullet$]
  $\mathcal{V}_{[0,T]} = \Big\{ Y=\varphi\big(\cdot,\eta(\cdot)\big) \big| \varphi\in C_{pol}^{1,3}([0,T]\times \mathbb{R})$  with $\frac{\partial \varphi}{\partial t}\in C_{pol}^{0,1}([0,T]\times \mathbb{R}), \ t\in[0,T] \Big\}.$
\end{itemize}
 Moreover, by $\widetilde{\mathcal{V}}_{[0,T+K]}$ and $\widetilde{\mathcal{V}}_{[0,T+K]}^{H}$ we denote the completion of $\mathcal{V}_{[0,T+K]}$ under the following norms respectively,
\begin{equation}\label{3.99}
  \| Y \| \triangleq \bigg(\dbE\int_0^{T+K}  e^{\beta t}  |Y(t)|^{2} dt\bigg)^{\frac{1}{2}}, \ \ \ \
  \| Z \| \triangleq \bigg(\dbE\int_0^{T+K} t^{2H-1} e^{\beta t} |Z(t)|^{2} dt\bigg)^{\frac{1}{2}},
\end{equation}
where $\beta\geq 0$ is a constant.
It is easy to see that  $\widetilde{\mathcal{V}}_{[0,T+K]}^{H} \subseteq \widetilde{\mathcal{V}}_{[0,T+K]} \subseteq L^{2}_{\mathcal{F}}(0,T+K;\mathbb{R})$.

\begin{definition}
We call $(Y, Z)$ a solution of BSDE (\ref{31}), if they  belong to
$\widetilde{\mathcal{V}}_{[0,T+K]}\times \widetilde{\mathcal{V}}^{H}_{[0,T+K]}$ and satisfy the equation (\ref{31}).
\end{definition}

The setting of our problem is as follows: to find a pair of processes
 $(Y_{\cdot},Z_{\cdot}) \in \widetilde{\mathcal{V}}_{[0,T+K]} \times \widetilde{\mathcal{V}}^{H}_{[0,T+K]}$ satisfying the BSDE (\ref{31}). In the following, we will use two different approaches to prove the existence and uniqueness of the equation (\ref{31}).

\subsection{The first approach}

In this subsection, the first approach, introduced by Maticiuc and Nie \cite{Maticiuc}, is used to establish the existence and uniqueness of  Eq. (\ref{31}).
In order to find the solution of BSDE (\ref{31}), the following assumptions are needed.
\begin{itemize}
\item[(H1)] $g$ and $h$ are given elements in $ C^{2}_{pol}(\mathbb{R})$ such that
\begin{equation*}
\mathbb{E}\int_T^{T+K} e^{\beta t}|g(\eta_{t})|^{2} dt< + \infty,\qq
  \mathbb{E}\int_T^{T+K} e^{\beta t}t^{2H-1}|h(\eta_{t})|^{2} dt< +\infty.
\end{equation*}

\end{itemize}

\begin{itemize}
\item[(H2)]
Assume that $f=f(t,x,y',z',y,z,\theta',\zeta',\theta,\zeta):[0,T]\times \mathbb{R}^{5}\times L^{2}(\mathcal{F}_{r'},\mathbb{R}) \times L^{2}(\mathcal{F}_{r},\mathbb{R})\times L^{2}(\mathcal{F}_{r'},\mathbb{R}) \times L^{2}(\mathcal{F}_{r},\mathbb{R})\longrightarrow L^{2}(\mathcal{F}_{t},\mathbb{R})$ is a $C_{pol}^{0,1}$-continuous function, where $r', r \in[t,T+K]$. Moreover, there is a constant $C\geq 0$ such that,
for every $t\in [0,T]$, $x,y,\bar{y},$ $z,\bar{z},y',\bar{y}',z',\bar{z}' \in \mathbb{R}$,
$\theta_{\cdot},\bar{\theta}_{\cdot},\theta'_{\cdot},\bar{\theta}'_{\cdot}$,
$\zeta_{\cdot},\bar{\zeta}_{\cdot},\zeta'_{\cdot},\bar{\zeta}'_{\cdot}\in L_{\mathcal{F}}^2(t,T+K;\mathbb{R})$,
we have
$$\ba{ll}
\ds |f(t,x,y',z',y,z,\theta'_{r'},\zeta'_{r},\theta_{r'},\zeta_{r})- f(t,x,\bar{y}',\bar{z}',\bar{y},\bar{z},\bar{\theta}'_{r'},\bar{\zeta}'_{r},\bar{\theta}_{r'},\bar{\zeta}_{r})|\\
  \leq C\bigg(|y'-\bar{y}'| +|z'-\bar{z}'| + |y-\bar{y}| +|z-\bar{z}| \\
\ns\ds\qq \ \ + \dbE'\bigg[  |\theta'_{r'}-\bar{\theta}'_{r'}| +|\zeta'_{r}-\bar{\zeta}'_{r}| \bigg| \mathcal{F}_{t}\bigg]
+ \dbE\bigg[ |\theta_{r'}-\bar{\theta}_{r'}| +|\zeta_{r}-\bar{\zeta}_{r}| \bigg| \mathcal{F}_{t} \bigg] \bigg).
\ea$$
\end{itemize}
For notational simplicity, we denote $f_{0}(t,x)=f_{0}(t,x,0,0,0,0,0,0,0,0)$.
\begin{theorem}\label{FirstMethod}
Under the assumptions (H1) and (H2),  BSDE (\ref{31}) admits a unique solution
$(Y_{\cdot},Z_{\cdot})\in\widetilde{\mathcal{V}}_{[0,T+K]} \times \widetilde{\mathcal{V}}_{[0,T+K]}^{H}$.
Moreover, for all $t\in[0,T]$,
\begin{equation}\label{20}
      \dbE\left( e^{\beta t}|Y_{t}|^{2} + \int_t^T e^{\beta s}s^{2H-1}|Z_{s}|^{2} ds\right)
 \leq R\Theta(t,T,K),
\end{equation}
where $R$ is a positive constant which may be different from line to line, and
\begin{equation*}
  \Theta(t,T,K)= \dbE\bigg(e^{\beta T}|g(\eta_{T})|^{2}  + \int_t^T e^{\beta s}|f_{0}(s,\eta_{s})|^{2} ds
                + \int_T^{T+K} e^{\beta s}\big(|g(\eta_{s})|^{2}+s^{2H-1}|h(\eta_{s})|^{2}\big) ds \bigg).
\end{equation*}
\end{theorem}

\begin{proof}
For any given $(y_{t},z_{t}) \in \widetilde{\mathcal{V}}_{[0,T+K]} \times \widetilde{\mathcal{V}}^{H}_{[0,T+K]}$, we consider the following simple BSDE:
\begin{equation}\label{32}
  \begin{cases}
    -dY_{t}=\dbE'[f(t,\eta_{t},y'_{t},z'_{t},y_{t},z_{t},y'_{t+\delta(t)},z'_{t+\zeta(t)},y_{t+\delta(t)},z_{t+\zeta(t)})] dt
            - Z_{t} dB_{t}^{H}, \ \ \ t\in [0,T]; \\
      Y_{t}=g(\eta_{t}), \ \  Z_{t}=h(\eta_{t}), \ \ \ t\in[T,T+K].
  \end{cases}
\end{equation}
From Proposition \ref{WS}, note that $Y_{t}=g(\eta_{t})$ and $Z_{t}=h(\eta_{t})$ are given when $t\in[T,T+K]$,
we obtain that BSDE (\ref{32}) has a unique solution
$(Y_{\cdot},Z_{\cdot}) \in \widetilde{\mathcal{V}}_{[0,T+K]} \times \widetilde{\mathcal{V}}^{H}_{[0,T+K]}$.

\ms

Define a mapping
$I:\widetilde{\mathcal{V}}_{[0,T+K]} \times \widetilde{\mathcal{V}}^{H}_{[0,T+K]}\longrightarrow
\widetilde{\mathcal{V}}_{[0,T+K]} \times \widetilde{\mathcal{V}}^{H}_{[0,T+K]}$
such that $I[(y_{\cdot},z_{\cdot})]=(Y_{\cdot},Z_{\cdot})$.
Due to the values of $Y_{t}$ and $Z_{t}$ are given when $t\in[T,T+K]$,
we essentially only need to prove (\ref{31}) has a unique solution on $[0,T]$.
Let $n\in \mathbb{N}$ and $t_{i}=\frac{i-1}{n}T, i=1,...,n+1$.
First we solve (\ref{31}) on $[t_{n},T]$.
In order to do this, we show $I$ is a contraction on $\widetilde{\mathcal{V}}_{[t_{n},T+K]} \times \widetilde{\mathcal{V}}^{H}_{[t_{n},T+K]}$.

\ms

For two arbitrary elements $(y_{\cdot},z_{\cdot})$ and
 $(\bar{y}_{\cdot},\bar{z}_{\cdot})\in \widetilde{\mathcal{V}}_{[t_{n},T+K]} \times \widetilde{\mathcal{V}}^{H}_{[t_{n},T+K]}$,
set $(Y_{\cdot},Z_{\cdot})=I[(y_{\cdot},z_{\cdot})]$ and $(\bar{Y}_{\cdot},\bar{Z}_{\cdot})=I[(\bar{y}_{\cdot},\bar{z}_{\cdot})]$.
We denote their differences by
\begin{equation*}
    (\hat{y}_{\cdot},\hat{z}_{\cdot})=(y_{\cdot}-\bar{y}_{\cdot},z_{\cdot}-\bar{z}_{\cdot}), \qq
  (\hat{Y}_{\cdot},\hat{Z}_{\cdot})=(Y_{\cdot}-\bar{Y}_{\cdot},Z_{\cdot}-\bar{Z}_{\cdot}).
\end{equation*}
By applying It\^{o} formula (Proposition \ref{3}), for $t\in [t_{n},T]$, one has
\bel{25}\ba{ll}
\ds e^{\beta t}\hat{Y}_{t}^{2} + \beta \int_t^T e^{\beta s}\hat{Y}_{s}^{2} ds
     + 2\int_t^T e^{\beta s}\mathbb{D}_{s}^{H} \hat{Y}_{s}\hat{Z}_{s} ds
     +2\int_t^T e^{\beta s}\hat{Y}_{s}\hat{Z}_{s} dB_{s}^{H} \\
\ns\ds =  2\int_t^T e^{\beta s}\hat{Y}_{s}
     \dbE'\big[f(s,\eta_{s},y'_{s},z'_{s},y_{s},z_{s},y'_{s+\delta(s)},z'_{s+\zeta(s)},y_{s+\delta(s)},z_{s+\zeta(s)})\\
\ns\ds\qq\qq\qq\ \ \ \ - f(s,\eta_{s},\overline{y}'_{s},\overline{z}'_{s},\overline{y}_{s},\overline{z}_{s},
\bar{y}'_{s+\delta(s)},\bar{z}'_{s+\zeta(s)},\bar{y}_{s+\delta(s)},\bar{z}_{s+\zeta(s)})\big] ds.
\ea\ee
We know (see Hu and Peng \cite{Hu}, Maticiuc and Nie \cite{Maticiuc}) that $\mathbb{D}_{s}^{H} \hat{Y}_{s} = \frac{\hat{\sigma}_{s}}{\sigma_{s}} \hat{Z}_{s}$.
Moreover, by Remark 6 in Maticiuc and Nie \cite{Maticiuc}, there is a constant $M>0$ such that
$$\frac{t^{2H-1}}{M}\leq \frac{\hat{\sigma}_{t}}{\sigma_{t}}\leq M t^{2H-1},\qq \forall t\in[0,T].$$
Without loss of generality, we can choose $M>2$ in the following discussion.
Then from (\ref{25}) and Proposition \ref{2}, we have
\bel{35}\ba{ll}
\ds \dbE\left(e^{\beta t}\hat{Y}_{t}^{2} + \beta \int_t^T e^{\beta s}\hat{Y}_{s}^{2} ds
     +  \frac{2}{M}\int_t^T e^{\beta s}s^{2H-1}\hat{Z}_{s}^{2} ds \right)\\
\ns\ds \leq   2\int_t^T e^{\beta s}\hat{Y}_{s}
     \dbE'\big[f(s,\eta_{s},y'_{s},z'_{s},y_{s},z_{s},y'_{s+\delta(s)},z'_{s+\zeta(s)},y_{s+\delta(s)},z_{s+\zeta(s)})\\
\ns\ds\qq\qq\qq \ \ \ \ - f(s,\eta_{s},\overline{y}'_{s},\overline{z}'_{s},\overline{y}_{s},\overline{z}_{s},
         \bar{y}'_{s+\delta(s)},\bar{z}'_{s+\zeta(s)},\bar{y}_{s+\delta(s)},\bar{z}_{s+\zeta(s)})\big] ds.
\ea\ee
From assumption (H2) and (\ref{34.2}) we obtain
\bel{87}\ba{ll}
\ds \dbE\left(e^{\beta t}\hat{Y}_{t}^{2} + \beta \int_t^T e^{\beta s}\hat{Y}_{s}^{2} ds
     +  \frac{2}{M}\int_t^T e^{\beta s}s^{2H-1}\hat{Z}_{s}^{2} ds \right) \\
\ns\ds\leq  2C \dbE\int_t^Te^{\beta s}|\hat{Y}_{s}|\dbE'\bigg(|\hat{y}'_{s}|+|\hat{z}'_{s}|
        + \dbE'\bigg[|\hat{y}'_{s+\delta(s)}|+|\hat{z}'_{s+\delta(s)}|\bigg| \mathcal{F}_{s} \bigg] \bigg) ds  \\
\ns\ds\ \ \ +2C \dbE\int_t^Te^{\beta s}|\hat{Y}_{s}|\dbE'\bigg(|\hat{y}_{s}|+|\hat{z}_{s}|
        + \dbE\bigg[|\hat{y}_{s+\delta(s)}|+|\hat{z}_{s+\delta(s)}|\bigg| \mathcal{F}_{s} \bigg] \bigg) ds \\
\ns\ds=4C \int_t^Te^{\beta s}\dbE\bigg(|\hat{Y}_{s}| \big(|\hat{y}_{s}|+|\hat{z}_{s}| \big) \bigg) ds
        +  4C \int_t^Te^{\beta s}\dbE\bigg( |\hat{Y}_{s}| \big(|\hat{y}_{s+\delta(s)}|+|\hat{z}_{s+\delta(s)}| \big) \bigg) ds. \ \
\ea\ee
Therefore by choosing $\beta\geq 1$, and using H\"{o}lder's inequality and Jensen's inequality we get
\bel{83}\ba{ll}
\ds \dbE\left(e^{\beta t}\hat{Y}_{t}^{2} + \int_t^T e^{\beta s}\hat{Y}_{s}^{2} ds
     +  \frac{2}{M}\int_t^T e^{\beta s}s^{2H-1}\hat{Z}_{s}^{2} ds \right)\\
\ns\ds \leq 4C\int_t^T \big(e^{\beta s}
\dbE|\hat{Y}_{s}|^{2}\big)^{\frac{1}{2}} \bigg(\big[e^{\beta s} \dbE(|\hat{y}_{s}|
+ |\hat{z}_{s}|)^{2}\big]^{\frac{1}{2}} + \big[e^{\beta s} \dbE(|\hat{y}_{s+\delta(s)}|
+ |\hat{z}_{s+\zeta(s)}|)^{2}\big]^{\frac{1}{2}} \bigg) ds. \
\ea\ee
Denote $x(t)=\big(e^{\beta t} E|\hat{Y}_{t}|^{2}\big)^{\frac{1}{2}}$. From (\ref{83}) we have
\begin{equation*}
  x(t)^{2}\leq  4C\int_t^T x(s)\bigg(\big[e^{\beta s} \dbE(|\hat{y}_{s}|
+ |\hat{z}_{s}|)^{2}\big]^{\frac{1}{2}} + \big[e^{\beta s} \dbE(|\hat{y}_{s+\delta(s)}|
+ |\hat{z}_{s+\zeta(s)}|)^{2}\big]^{\frac{1}{2}} \bigg) ds.
\end{equation*}
Applying Lemma 20 in Maticiuc and Nie \cite{Maticiuc} to the above inequality one has
$$\ba{ll}
\ds x(t)\leq 2C\int_t^T \bigg(\big[e^{\beta s} \dbE(|\hat{y}_{s}|
+ |\hat{z}_{s}|)^{2}\big]^{\frac{1}{2}} + \big[e^{\beta s} \dbE(|\hat{y}_{s+\delta(s)}|
+ |\hat{z}_{s+\zeta(s)}|)^{2}\big]^{\frac{1}{2}} \bigg) ds\\
\ns\ds\qq \leq 2\sqrt{2}C \int_t^T \bigg(e^{\beta s} \dbE\big(|\hat{y}_{s}|^{2} + |\hat{z}_{s}|^{2}\big)\bigg)^{\frac{1}{2}} ds+ 2\sqrt{2}C \int_t^T \bigg(e^{\beta s} \dbE\big(|\hat{y}_{s+\delta(s)}|^{2}
+ |\hat{z}_{s+\zeta(s)}|^{2}\big)\bigg)^{\frac{1}{2}} ds.\ \ \ \ \
\ea$$
Therefore for $t\in [t_{n},T]$,
$$x(t)^{2}\leq 16C^{2} \bigg(\int_t^T \big[e^{\beta s} \dbE(|\hat{y}_{s}|^{2} + |\hat{z}_{s}|^{2})\big]^{\frac{1}{2}} ds\bigg)^{2} + 16C^{2} \bigg(\int_t^T \big[e^{\beta s} \dbE(|\hat{y}_{s+\delta(s)}|^{2}
 + |\hat{z}_{s+\zeta(s)}|^{2})\big]^{\frac{1}{2}} ds \bigg)^{2}.$$
Now we compute
\bel{84}\ba{ll}
\ds \int_{t_{n}}^T x(s)^{2} ds
  \leq 16C^{2}(T-t_{n}) \bigg(\int_{t_{n}}^T \big[e^{\beta s} \dbE(|\hat{y}_{s}|^{2} + |\hat{z}_{s}|^{2})\big]^{\frac{1}{2}} ds\bigg)^{2} \\
\ns\ds\qq\qq\qq \ +16C^{2}(T-t_{n})\bigg(\int_{t_{n}}^T \big[e^{\beta s} \dbE(|\hat{y}_{s+\delta(s)}|^{2}
       +|\hat{z}_{s+\zeta(s)}|^{2})\big]^{\frac{1}{2}}ds\bigg)^{2} \qq \ \\
\ns\ds\qq\qq\ \ \ \ =: A_{1} + A_{2}.
\ea\ee
For  the term $A_{2}$ of (\ref{84}) we deduce
\bel{85}\ba{ll}
\ds \bigg(\int_{t_{n}}^T \big[e^{\beta s} \dbE(|\hat{y}_{s+\delta(s)}|^{2}+|\hat{z}_{s+\zeta(s)}|^{2})\big]^{\frac{1}{2}}ds\bigg)^{2} \\
\ns\ds \leq \bigg(\int_{t_{n}}^T \big[e^{\beta s} \dbE|\hat{y}_{s+\delta(s)}|^{2} \big]^{\frac{1}{2}} ds
      +\int_{t_{n}}^T \big[e^{\beta s} \dbE|\hat{z}_{s+\zeta(s)}|^{2}\big]^{\frac{1}{2}}ds\bigg)^{2} \\
\ns\ds \leq 2\bigg(\int_{t_{n}}^T \big[e^{\beta s} \dbE|\hat{y}_{s+\delta(s)}|^{2} \big]^{\frac{1}{2}} ds \bigg)^{2}
      +2\bigg(\int_{t_{n}}^T \big[\frac{1}{s^{2H-1}} \cdot e^{\beta s} s^{2H-1}\dbE|\hat{z}_{s+\zeta(s)}|^{2}\big]^{\frac{1}{2}}ds\bigg)^{2}\\
\ns\ds \leq 2(T-t_{n})\int_{t_{n}}^Te^{\beta s} \dbE|\hat{y}_{s+\delta(s)}|^{2} ds
      +\frac{2(T^{2-2H}-t_{n}^{2-2H})}{2-2H} \int_{t_{n}}^Te^{\beta s} s^{2H-1}\dbE|\hat{z}_{s+\zeta(s)}|^{2} ds\\
\ns\ds \leq \Big(2(T-t_{n})+ \frac{T^{2-2H}-t_{n}^{2-2H}}{1-H}\Big)
      \dbE \int_{t_{n}}^T \big[e^{\beta (s+\delta(s))}|\hat{y}_{s+\delta(s)}|^{2}
       + e^{\beta (s+\zeta(s))}(s+\zeta(s))^{2H-1}|\hat{z}_{s+\zeta(s)}|^{2}\big] ds \\
\ns\ds \leq \Big(2(T-t_{n})+ \frac{T^{2-2H}-t_{n}^{2-2H}}{1-H}\Big)L\cdot
      \dbE \int_{t_{n}}^{T+K} e^{\beta s}\big(|\hat{y}_{s}|^{2} + s^{2H-1}|\hat{z}_{s}|^{2}\big) ds.
\ea\ee
In the last inequality, we used the condition (ii) satisfied by $\delta(\cd)$ and $\zeta(\cd)$.
Similarly, for $A_{1}$ of (\ref{84}),
\bel{303}\ba{ll}
\ds \bigg(\int_{t_{n}}^T \big[e^{\beta s}\mathbb{E}
(|\hat{y}_{s}|^{2} + |\hat{z}_{s}|^{2})\big]^{\frac{1}{2}} ds\bigg)^{2}\\
\ns\ds \leq\big[2(T-t_{n})+ \frac{T^{2-2H}-t_{n}^{2-2H}}{1-H}\big]
      \dbE\int_{t_{n}}^{T} e^{\beta s}\big(|\hat{y}_{s}|^{2} + s^{2H-1}|\hat{z}_{s}|^{2}\big) ds\\
\ns\ds \leq \big[2(T-t_{n})+ \frac{T^{2-2H}-t_{n}^{2-2H}}{1-H}\big]
      \dbE \int_{t_{n}}^{T+K} e^{\beta s}\big(|\hat{y}_{s}|^{2} + s^{2H-1}|\hat{z}_{s}|^{2}\big) ds. \qq \
\ea\ee
Combining (\ref{84}-\ref{303}), it follows that
\begin{equation}\label{304}
  \int_{t_{n}}^T x(s)^{2} ds
  \leq(T-t_{n})G\cdot \dbE \int_{t_{n}}^{T+K} e^{\beta s}\big(|\hat{y}_{s}|^{2} + s^{2H-1}|\hat{z}_{s}|^{2}\big) ds,
  \qq\qq \ \ \
\end{equation}
where $G=16C^{2} (L+1)\big[2(T-t_{n})+ \frac{T^{2-2H}-t_{n}^{2-2H}}{1-H}\big]$. And similarly one has
\begin{equation}\label{305}
  \int_{t_{n}}^T \frac{1}{s^{2H-1}}x(s)^{2} ds
  \leq G \frac{T^{2-2H}-t_{n}^{2-2H}}{2-2H} \mathbb{E} \int_{t_{n}}^{T+K} e^{\beta s}\big(|\hat{y}_{s}|^{2} + s^{2H-1}|\hat{z}_{s}|^{2}\big) ds.
\end{equation}
Now from (\ref{87}),
$$\ba{ll}
\ds  \dbE\left( \int_{t_{n}}^T e^{\beta s}|\hat{Y}_{s}|^{2} ds
       + \frac{2}{M}\int_{t_{n}}^T e^{\beta s}s^{2H-1}|\hat{Z}_{s}|^{2} ds\right)\\
\ns\ds \leq  4C \dbE\int_{t_{n}}^T e^{\beta s} \bigg(\frac{1}{v}(1+\frac{1}{s^{2H-1}})|\hat{Y}_{s}|^{2}
        + v|\hat{y}_{s}|^{2} + vs^{2H-1}|\hat{z}_{s}|^{2}  \bigg) ds \\
\ns\ds\ \ \ + 4C \dbE\int_{t_{n}}^T e^{\beta s} \bigg(\frac{1}{v}\big(1+\frac{1}{s^{2H-1}}\big)|\hat{Y}_{s}|^{2}
        + v|\hat{y}_{s+\delta(s)}|^{2} + vs^{2H-1}|\hat{z}_{s+\zeta(s)}|^{2}  \bigg) ds \\
\ns\ds \leq  \frac{8C}{v} \dbE\int_{t_{n}}^T e^{\beta s}(1+\frac{1}{s^{2H-1}})|\hat{Y}_{s}|^{2} ds
        + 4C v \dbE\int_{t_{n}}^T e^{\beta s}\big(|\hat{y}_{s}|^{2} + s^{2H-1}|\hat{z}_{s}|^{2}\big) ds \\
\ns\ds\ \ \ +4Cv\dbE\int_{t_{n}}^T e^{\beta s}\big(|\hat{y}_{s+\delta(s)}|^{2}
+ s^{2H-1}|\hat{z}_{s+\zeta(s)}|^{2}\big)ds\\
\ns\ds \leq\frac{8C}{v} \dbE\int_{t_{n}}^T e^{\beta s}(1+\frac{1}{s^{2H-1}})|\hat{Y}_{s}|^{2} ds
        + 4C v(1+L) \dbE\int_{t_{n}}^{T+K} e^{\beta s}\big(|\hat{y}_{s}|^{2} + s^{2H-1}|\hat{z}_{s}|^{2}\big) ds,
\ea$$
where $v>0$. Using the inequalities (\ref{304}) and (\ref{305}), and note that $M>2$, we obtain
\begin{equation*}
  \dbE\left( \int_{t_{n}}^T e^{\beta s}|\hat{Y}_{s}|^{2} ds + \int_{t_{n}}^T e^{\beta s}s^{2H-1}|\hat{Z}_{s}|^{2} ds\right)
\leq \widetilde{G} \dbE\int_{t_{n}}^{T+K} e^{\beta s}\big(|\hat{y}_{s}|^{2} + s^{2H-1}|\hat{z}_{s}|^{2}\big) ds,
\end{equation*}
or
\begin{equation*}
  \dbE \int_{t_{n}}^{T+K} e^{\beta s}\big(|\hat{Y}_{s}|^{2} + s^{2H-1}|\hat{Z}_{s}|^{2}\big) ds
\leq \widetilde{G} \dbE\int_{t_{n}}^{T+K} e^{\beta s}\big(|\hat{y}_{s}|^{2} + s^{2H-1}|\hat{z}_{s}|^{2}\big) ds,
 \qq\ \ \ \
\end{equation*}
where $$\widetilde{G}=\frac{4CGM}{v}(T-t_{n}) + \frac{4CGM}{v(1-H)}(T^{2-2H}-t_{n}^{2-2H}) + 2CM(1+L)v.\qq\qq\qq$$
Choosing $v$ such that $2CM(1+L)v<\frac{1}{4}$,
and taking $n$ large enough such that
$$\frac{4CGM}{v}(T-t_{n})<\frac{1}{4}, \qq \frac{4CGM}{v(1-H)}(T^{2-2H}-t_{n}^{2-2H})<\frac{1}{4},\qq\qq\qq\qq\qq\ $$
then
\begin{equation*}
  \dbE \int_{t_{n}}^{T+K} e^{\beta s}\big(|\hat{Y}_{s}|^{2} + s^{2H-1}e^{\beta s}|\hat{Z}_{s}|^{2}\big) ds
\leq \frac{3}{4} \dbE\int_{t_{n}}^{T+K} e^{\beta s}\big(|\hat{y}_{s}|^{2} + s^{2H-1}|\hat{z}_{s}|^{2}\big) ds.\qq
\end{equation*}
Hence $I$ is a contraction on $\widetilde{\mathcal{V}}_{[t_{n},T+K]} \times \widetilde{\mathcal{V}}^{H}_{[t_{n},T+K]}$, which implies that BSDE (\ref{31}) has a unique solution on $[t_{n},T]$.
The next step is to solve (\ref{31}) on  $[t_{n-1},t_{n}].$
In order to do this, one can show $I$ is a contraction on $\widetilde{\mathcal{V}}_{[t_{n-1},t_{n}+K]} \times \widetilde{\mathcal{V}}^{H}_{[t_{n-1},t_{n}+K]}$.
With the same arguments, repeating the above technique we obtain that  BSDE (\ref{31}) admits a unique solution on
 $\widetilde{\mathcal{V}}_{[0,T+K]} \times \widetilde{\mathcal{V}}_{[0,T+K]}^{H}$.

\ms

Now we prove the estimate (\ref{20}).
Suppose $(Y,Z)$ is the solution of BSDE (\ref{31}).
From (H2), similarly to (\ref{35}) we obtain
$$\ba{ll}
\ds \dbE\left(e^{\beta t}Y_{t}^{2} + \beta \int_t^T e^{\beta s}Y_{s}^{2} ds
     +  \frac{2}{M}\int_t^T e^{\beta s}s^{2H-1}Z_{s}^{2} ds \right) \\
\ns\ds \leq  \dbE\bigg(e^{\beta T}|g(\eta_{T})|^{2}+  2\int_t^T e^{\beta s}Y_{s} \dbE'[f(s,\eta_{s},Y'_{s},Z'_{s},Y_{s},Z_{s},
Y'_{s+\delta(s)},Z'_{s+\zeta(s)},Y_{s+\delta(s)},Z_{s+\zeta(s)})] ds\bigg).
\ea$$
By Lipschitz continuity of $f$, similar as the above discussion, we have
$$\ba{ll}
\ds  2\dbE\int_t^T e^{\beta s} Y_{s} \dbE'[f(s,\eta_{s},Y'_{s},Z'_{s},Y_{s},Z_{s},Y'_{s+\delta(s)},Z'_{s+\zeta(s)},Y_{s+\delta(s)},Z_{s+\zeta(s)})] ds \\
\ns\ds \leq 4\dbE\int_t^T e^{\beta s}|Y_{s}|\bigg(C\big(|Y_{s}|+|Z_{s}|+|Y_{s+\delta(s)}|
+|Z_{s+\zeta(s)}|\big)+|f_{0}(s,\eta_{s})|\bigg) ds  \\
\ns\ds \leq \dbE\int_t^T 4\bigg(2C+C^{2}+\frac{2C^{2}M}{s^{2H-1}}+\frac{2C^{2}ML}{s^{2H-1}}+1\bigg)e^{\beta s}|Y_{s}|^{2}ds + \frac{1}{2M}\dbE\int_t^T e^{\beta s}s^{2H-1}|Z_{s}|^{2}ds  \\
\ns\ds\ \ \ + \dbE\int_t^T e^{\beta s}|Y_{s+\delta(s)}|^{2}ds
       + \frac{1}{2ML}\dbE\int_t^T e^{\beta s}s^{2H-1}|Z_{s+\zeta(s)}|^{2}ds
       + 4\dbE\int_t^T e^{\beta s}|f_{0}(s,\eta_{s})|^{2}ds\\
\ns\ds \leq \dbE\int_t^T 4\bigg(2C+C^{2}+\frac{2C^{2}M}{s^{2H-1}}+\frac{2C^{2}ML}{s^{2H-1}}+1+L\bigg)e^{\beta s}|Y_{s}|^{2}ds + \frac{1}{M}\dbE\int_t^T e^{\beta s}s^{2H-1}|Z_{s}|^{2}ds \\
\ns\ds\ \ \ +L\dbE\int_T^{T+K} e^{\beta s}|g(\eta_{s})|^{2}ds
       + \frac{1}{2M}\dbE\int_T^{T+K} e^{\beta s}s^{2H-1}|h(\eta_{s})|^{2} dt
       + 4\dbE\int_t^T e^{\beta s}|f_{0}(s,\eta_{s})|^{2}ds. \ \ \ \
\ea$$
Thus, we have
\bel{308}\ba{ll}
\ds \dbE\left(e^{\beta t}|Y_{t}|^{2} + \frac{1}{M}\int_t^T e^{\beta s}s^{2H-1}|Z_{s}|^{2} ds\right)\\
\ns\ds \leq  R\Theta(t,T,K) + \dbE\int_t^T 4\bigg(2C+C^{2}+L+1+\frac{2C^{2}M(L+1)}{s^{2H-1}}\bigg)e^{\beta s}|Y_{s}|^{2}ds.
\ea\ee
By Gronwall's inequality,
\begin{equation*}
   e^{\beta t}\dbE|Y_{t}|^{2} \leq R\Theta(t,T,K)\exp \bigg\{ 4(2C+C^{2}+L+1)(T-t) + 8C^{2}M(L+1)\frac{T^{2-2H}-t^{2-2H}}{2-2H} \bigg\}.
\end{equation*}
Finally, from (\ref{308}), combining the above estimate one  has
\begin{equation*}
   \dbE\int_t^T e^{\beta s}s^{2H-1}|Z_{s}|^{2} ds \leq R\Theta(t,T,K).
\end{equation*}
Hence the estimate (\ref{20}) is obtained.
This completes the proof.
\end{proof}

\begin{remark}
In the proof of Theorem \ref{FirstMethod}, we first divide the interval $[0,T+K]$, and then we prove BSDE \rf{31} has unique solution in each subinterval of $[0,T+K]$. Next, we use another approach to directly prove that BSDE \rf{31} admits unique solution in $[0,T+K]$.
\end{remark}

\subsection{The second approach}
In this section we present the second approach to prove the existence and uniqueness of solutions of equation (\ref{31}).
It should be pointed out that this approach is more convenient than the above one.
However, the price of doing this is that we should strengthen the condition of the coefficient $f$ with respect to $z$.

\begin{itemize}
\item[(H3)]
Assume that $f=f(t,x,y',z',y,z,\theta',\zeta',\theta,\zeta):[0,T]\times \mathbb{R}^{5}\times L^{2}(\mathcal{F}_{r'},\mathbb{R}) \times L^{2}(\mathcal{F}_{r},\mathbb{R})\times L^{2}(\mathcal{F}_{r'},\mathbb{R}) \times L^{2}(\mathcal{F}_{r},\mathbb{R})\longrightarrow L^{2}(\mathcal{F}_{t},\mathbb{R})$ is a $C_{pol}^{0,1}$-continuous function, where $r', r \in[t,T+K]$. Moreover, there is a constant $C\geq 0$ such that,
for every $t\in [0,T]$, $x,y,\bar{y},z,\bar{z},y',\bar{y}',$ $z',\bar{z}' \in \mathbb{R}$,
$\theta_{\cdot},\bar{\theta}_{\cdot},\theta'_{\cdot},\bar{\theta}'_{\cdot}$,
$\zeta_{\cdot},\bar{\zeta}_{\cdot},\zeta'_{\cdot},\bar{\zeta}'_{\cdot}\in L_{\mathcal{F}}^2(t,T+K;\mathbb{R})$,
we have
$$\ba{ll}
\ds |f(t,x,y',z',y,z,\theta'_{r'},\zeta'_{r},\theta_{r'},\zeta_{r})
-f(t,x,\bar{y}',\bar{z}',\bar{y},\bar{z},\bar{\theta}'_{r'},,\bar{\zeta}'_{r},\bar{\theta}_{r'},\bar{\zeta}_{r})|\\
\ns\ds\leq C\bigg(|y'-\bar{y}'| +t^{H-\frac{1}{2}}|z'-\bar{z}'| + |y-\bar{y}| + t^{H-\frac{1}{2}}|z-\bar{z}| \\
\ns\ds\qq \ \ + \dbE'\bigg[  |\theta'_{r'}-\bar{\theta}'_{r'}| + r^{H-\frac{1}{2}}|\zeta'_{r}-\bar{\zeta}'_{r}| \bigg| \mathcal{F}_{t}\bigg]+ \dbE\bigg[ |\theta_{r'}-\bar{\theta}_{r'}| + r^{H-\frac{1}{2}}|\zeta_{r}-\bar{\zeta}_{r}| \bigg| \mathcal{F}_{t} \bigg] \bigg).
\ea$$
\end{itemize}

\begin{remark}\label{36}
Suppose $ \theta'$ is a square integrable, jointly measurable stochastic process.
Then we can define for all $t\in [0,T]$, $x,y,z\in \mathbb{R}$, $\theta'\in L^{2}(\mathcal{F}_{r'},\mathbb{R})$,
\begin{equation*}
  f^{\theta'}(t,x,y,z)\triangleq \dbE'[f(t,x,y,z,\theta'_{t+\delta(t)})]
   = \int_{\Omega} f(t,x,y,z,\theta'_{t+\delta(t)}(\omega')) \dbP (d\omega').
\end{equation*}
Indeed, due to the assumption on the coefficient $f$ being $C_{pol}^{0,1}$-continuous,
we know that $f^{\theta'}$ is also $C_{pol}^{0,1}$-continuous.
In addition, with the same constant $C$ of assumption (H3), for every $t\in [0,T]$, $x,y,\bar{y},z,\bar{z}
\in \mathbb{R}$, we have
\begin{equation*}
  |f^{\theta'}(t,x,y,z) - f^{\theta'}(t,x,\bar{y},\bar{z})|\leq C\big(|y_{1}-y_{2}| + t^{H-\frac{1}{2}}|z_{1}-z_{2}| \big).
\end{equation*}
\end{remark}
This remark is useful in the proof of the comparison theorem (see Section 4). Now, we show the existence and uniqueness theorem.

\begin{theorem}\label{th}
Under the assumptions (H1) and (H3),  BSDE (\ref{31}) admits a unique solution in $\widetilde{\mathcal{V}}_{[0,T+K]} \times \widetilde{\mathcal{V}}_{[0,T+K]}^{H}$.
\end{theorem}

\begin{proof}
Firstly, similar to the previous approach, for any given $(y_{t},z_{t}) \in \widetilde{\mathcal{V}}_{[0,T+K]} \times \widetilde{\mathcal{V}}^{H}_{[0,T+K]}$, we consider the following BSDE:
$$\left\{\ba{ll}
\ds -dY_{t}= \dbE'[f(t,\eta_{t},{y}^{\prime}_{t},{z}^{\prime}_{t},y_{t},z_{t},{y}^{\prime}_{t+\delta(t)}, {z}^{\prime}_{t+ \zeta(t)},y_{t+\delta(t)},z_{t+ \zeta(t)})] dt - Z_{t} dB_{t}^{H}, \qq t\in [0,T]; \\
\ns\ds Y_{t}=g(\eta_{t}), \qq Z_{t}=h(\eta_{t}), \qq  t\in[T,T+K].
\ea\right.$$
Define a mapping
$I:\widetilde{\mathcal{V}}_{[0,T+K]} \times \widetilde{\mathcal{V}}^{H}_{[0,T+K]}\longrightarrow \widetilde{\mathcal{V}}_{[0,T+K]} \times \widetilde{\mathcal{V}}^{H}_{[0,T+K]}$
such that $I[(y_{\cdot},z_{\cdot})]=(Y_{\cdot},Z_{\cdot})$.
Now we show that $I$ is a contraction mapping.
For two arbitrary elements $(y_{\cdot},z_{\cdot})$ and
 $(\bar{y}_{\cdot},\bar{z}_{\cdot})\in \widetilde{\mathcal{V}}_{[0,T+K]} \times \widetilde{\mathcal{V}}^{H}_{[0,T+K]}$,
set $(Y_{\cdot},Z_{\cdot})=I[(y_{\cdot},z_{\cdot})]$ and $(\bar{Y}_{\cdot},\bar{Z}_{\cdot})=I[(\bar{y}_{\cdot},\bar{z}_{\cdot})]$.
We denote their differences by
\begin{equation*}
    (\hat{y}_{\cdot},\hat{z}_{\cdot})=(y_{\cdot}-\bar{y}_{\cdot},z_{\cdot}-\bar{z}_{\cdot}), \qq
  (\hat{Y}_{\cdot},\hat{Z}_{\cdot})=(Y_{\cdot}-\bar{Y}_{\cdot},Z_{\cdot}-\bar{Z}_{\cdot}).
\end{equation*}
By the estimate (\ref{34}) we have
$$\ba{ll}
\ds \dbE\int_0^T e^{\beta s}\Big(\frac{\beta}{2}|\hat{Y}_{s}|^{2}  + \frac{2}{M}s^{2H-1}|\hat{Z}_{s}|^{2} \Big)ds\\
\ns\ds \leq \frac{2}{\beta} {\dbE}\int_0^T e^{\beta s}
      \Big|\dbE'\big[
      f(s,\eta_{s},y'_s,z'_s,y_s,z_s,y'_{s+\delta(s)},z'_{s+\delta(s)},y_{s+\delta(s)},z_{s+\delta(s)}) \\
\ns\ds\qq\qq\qq \ \ \ \ \ -
 f(s,\eta_{s},\bar{y}'_s,\bar{z}'_s,\bar{y}_s,\bar{z}_s,\bar{y}'_{s+\delta(s)},\bar{z}'_{s+\delta(s)},
 \bar{y}_{s+\delta(s)},\bar{z}_{s+\delta(s)})
  \big]\Big|^{2} ds.
\ea$$
From assumption (H3), Jensen's inequality and \rf{34.2} we obtain
$$\ba{ll}
\ds  {\dbE}\bigg[\Big|\dbE'\big[
      f(s,\eta_{s},y'_s,z'_s,y_s,z_s,y'_{s+\delta(s)},z'_{s+\delta(s)},y_{s+\delta(s)},z_{s+\delta(s)}) \\
\ns\ds\qq \ \ \ -
 f(s,\eta_{s},\bar{y}'_s,\bar{z}'_s,\bar{y}_s,\bar{z}_s,\bar{y}'_{s+\delta(s)},\bar{z}'_{s+\delta(s)},
 \bar{y}_{s+\delta(s)},\bar{z}_{s+\delta(s)})
  \big]\Big|^2\bigg]  \\
\ns\ds\leq \dbE\bigg[\dbE'\Big[\Big|
      f(s,\eta_{s},y'_s,z'_s,y_s,z_s,y'_{s+\delta(s)},z'_{s+\delta(s)},y_{s+\delta(s)},z_{s+\delta(s)}) \\
\ns\ds\qq\qq -
 f(s,\eta_{s},\bar{y}'_s,\bar{z}'_s,\bar{y}_s,\bar{z}_s,\bar{y}'_{s+\delta(s)},\bar{z}'_{s+\delta(s)},
 \bar{y}_{s+\delta(s)},\bar{z}_{s+\delta(s)})
\Big|^2\Big]\bigg]  \\
\ns\ds \leq C^{2} {\dbE}\bigg[{\dbE}^{\prime}\Big[ \Big( |\hat{y}'_s| + |\hat{y}_s| +s^{H-\frac{1}{2}}|\hat{z}'_s|+ s^{H-\frac{1}{2}}|\hat{z}_s| + {\dbE}^{\prime}\big[|\hat{y}'_{s+\delta(s)}|\big|\mathcal{F}_{t}\big]+ {\dbE}\big[|\hat{y}_{s+\delta(s)}|\big|\mathcal{F}_{t}\big]\\
\ns\ds\qq\qq \ \ \ \ \ + \big(s+\zeta(s)\big)^{H-\frac{1}{2}} {\dbE}^{\prime}\big[|\hat{z}'_{s+\delta(s)}|\big|\mathcal{F}_{t}\big]
+ \big(s+\zeta(s)\big)^{H-\frac{1}{2}}{\dbE}\big[|\hat{z}_{s+\delta(s)}|\big|\mathcal{F}_{t}\big]\Big)^{2}\Big]\bigg] \\
\ns\ds  \leq  16 C ^{2} \dbE\bigg[|\hat{y}_s|^{2}
+ s^{2H-1} |\hat{z}_s|^{2} + |\hat{y}_{s+\delta(s)}|^{2} + \big(s+\zeta(s)\big)^{2H-1}|\hat{z}_{s+\zeta(s)}|^{2} \bigg],
\ea$$
where we used the notation $(\hat{y}'_{\cdot},\hat{z}'_{\cdot})=(y'_{\cdot}-\bar{y}'_{\cdot},z'_{\cdot}-\bar{z}'_{\cdot})$ and the fact that $(a+b)^{2} \leq 2 a^{2} + 2 b^{2}$.  Then note that $\delta$ and $\zeta$ satisfy (i) and (ii), we obtain
$$\ba{ll}
\ds{\dbE}\int_0^T e^{\beta s}\Big(\frac{\beta}{2}|\hat{Y}_{s}|^{2}  + \frac{2}{M}s^{2H-1}|\hat{Z}_{s}|^{2} \Big)ds\\
\ns\ds \leq \frac{32 C^{2}}{\beta} {\dbE}\int_0^T e^{\beta s} \Big( |\hat{y}_s|^{2}
+ s^{2H-1} |\hat{z}_s|^{2} + |\hat{y}_{s+\delta(s)}|^{2} + \big(s+\zeta(s)\big)^{2H-1}|\hat{z}_{s+\zeta(s)}|^{2}\Big) ds\qq \\
\ns\ds \leq \frac{32 C^{2} (L+1)}{\beta} {\dbE}\int_{0}^{T+K} e^{\beta s} \Big(|\hat{y}_s|^{2}
+ s^{2H-1} |\hat{z}_s|^{2} \Big) ds.
\ea$$
Therefore one has
\begin{equation*}
{\dbE}\int_0^T e^{\beta s}\Big(\frac{M \beta}{4}|\hat{Y}_{s}|^{2}  + s^{2H-1}|\hat{Z}_{s}|^{2} \Big)ds
\leq \frac{16 C^{2}(L+1)M}{\beta} {\dbE}\int_{0}^{T+K} e^{\beta s} \Big(|\hat{y}_s|^{2}
+ s^{2H-1} |\hat{z}_s|^{2} \Big) ds.
\end{equation*}
Finally, by letting $\beta = 32 C^{2}(L+1)M + \frac{4}{M}$ we get
   \begin{equation*}
 {\dbE}\int_0^{T+K} e^{\beta s}\Big(|\hat{Y}_{s}|^{2}  + s^{2H-1}|\hat{Z}_{s}|^{2} \Big)ds
\leq \frac{1}{2} \mathbb{E}\int_{0}^{T+K} e^{\beta s} \Big(|\hat{y}_s|^{2}
+ s^{2H-1} |\hat{z}_s|^{2} \Big) ds.
 \end{equation*}
Consequently,
$I$ is a contraction on $\widetilde{\mathcal{V}}_{[0,T+K]} \times \widetilde{\mathcal{V}}^{H}_{[0,T+K]}$.
It follows by the fixed point theorem that BSDE (\ref{31}) has a unique solution
 in $\widetilde{\mathcal{V}}_{[0,T+K]} \times \widetilde{\mathcal{V}}^{H}_{[0,T+K]}$.
\end{proof}

\begin{remark}
Now, we make a comparison between the above two approaches.
It is easy to see that (H2) is weaker than (H3).
So from the point of view of conditions, the first approach is better than the second one.
On the other hand, thanks to the concise proof, the second approach is convenient than the first one.
So from this point of view, the second approach is better.
\end{remark}

\section{Comparison theorem}

In this section, we study a comparison theorem of MF-ABSDEs of the following form:
\bel{41}\left\{\ba{ll}
\ds-dY_{t}= \dbE'[f(t,\eta_{t},Y_{t},Z_{t},Y'_{t+\delta(t)})] dt - Z_{t} dB_{t}^{H}, \qq t\in [0,T]; \\
\ns\ds Y_{t}=g(\eta_{t}),  \qq t\in[T,T+K].
\ea\right.\ee
Under (H1) and (H3), it is easy to know that the above equation admits a unique solution.
Here, not (H2), we use (H3) because it is more convenient for the proof of the following comparison theorem.
\begin{theorem}\label{40}
For $i=1,2$, suppose $g_{i}$ satisfies (H1), and $f_{i}$ and $\partial_{\theta'}f_{i}$ satisfy (H3).
Moreover, assume
$f_1(t,x,y,z,\cdot)$ is increasing, i.e.,
$f_1(t,x,y,z,\theta'_{r})\leq f_1(t,x,y,z,\bar{\theta}'_{r})$,
if $\theta'_{r}\leq \bar{\theta}'_{r}$, $\theta'_{r},\bar{\theta}'_{r}\in L^{2}_{\mathcal{F}}(t,T+K;\mathbb{R})$, $r\in [t,T+K]$.
Then, if $g_{1}(x)\leq g_{2}(x)$ and $f_{1}(t,x,y,z,\theta')\leq f_{2}(t,x,y,z,\theta')$ for all $(t,x,y,z)\in[0,T]\times \mathbb{R}^{3}$, $\theta'_1\in L^{2}(\mathcal{F}_{r},\mathbb{R})$,
we have $Y_{1}(t)\leq Y_{2}(t)$ almost surely.
\end{theorem}
\begin{proof}
For $i=1,2$, we define $f_{i}^{\theta'}(s,x,y,z) \triangleq \dbE'[f_{i}(s,x,y,z,\theta'_{s+\delta(s)})]$.
By virtue of Remark \ref{36}, we see that $f^{\theta'}_{i}$ and $\partial_{\theta'}f^{\theta'}_{i}$ satisfy (H3).
In addition,  $f^{\theta'}_{1}$ is increasing in $\theta'$ and $f^{\theta'}_{1}\leq f^{\theta'}_{2}$.

\ms

Let $\widetilde{Y}_{0}(\cdot)=Y_{2}(\cdot)$. We consider the following BSDE,
$$\left\{\ba{ll}
\ds \widetilde{Y}_{1}(t)=g_{1}(\eta_{T}) + \int_t^T \dbE'[f_{1}(s,\eta_{s},\widetilde{Y}_{1}(s),\widetilde{Z}_{1}(s),\widetilde{Y}'_{0}(s+\delta(s)))] ds -\int_t^T \widetilde{Z}_{1}(s) dB_{s}^{H}, \qq t\in[0,T];\\
\ns\ds Y_{1}(t)=g_1(\eta_{t}),\qq t\in[T,T+K].
\ea\right.$$
By Theorem \ref{th}, the above equation admits a unique solution
$(\widetilde{Y}_{1}(\cdot),\widetilde{Z}_{1}(\cdot)) \in \widetilde{\mathcal{V}}_{[0,T+K]} \times \widetilde{\mathcal{V}}^{H}_{[0,T]}$.
Now based on the assumptions, we have
\begin{equation*}
  \begin{cases}
   f_{1}^{\widetilde{Y}'_{0}}(t,x,y,z) \leq  f_{2}^{\widetilde{Y}'_{0}}(t,x,y,z), \qq \forall (t,x,y,z)\in [0,T]\times \mathbb{R}^{3};\\
   g_{1}(x)\leq g_{2}(x), \qq \forall x\in \mathbb{R}.
  \end{cases}
\end{equation*}
So from Theorem 12.3 of Hu et al. \cite{Hu1}, we deduce
\begin{equation*}
  \widetilde{Y}_{1}(t)\leq \widetilde{Y}_{0}(t)=Y_{2}(t), \ \ a.s.
\end{equation*}
Next, we consider the following BSDE,
$$\left\{\ba{ll}
\ds \widetilde{Y}_{2}(t)=g_{1}(\eta_{T}) + \int_t^T \dbE'[f_{1}(s,\eta_{s},\widetilde{Y}_{2}(s),\widetilde{Z}_{2}(s),\widetilde{Y}'_{1}(s+\delta(s)))] ds -\int_t^T \widetilde{Z}_{2}(s) dB_{s}^{H}, \qq t\in[0,T];\\
\ns\ds Y_{2}(t)=g_1(\eta_{t}),  \qq t\in[T,T+K].
\ea\right.$$
 And denote by $(\widetilde{Y}_{2}(\cdot),\widetilde{Z}_{2}(\cdot)) \in \widetilde{\mathcal{V}}_{[0,T+K]} \times \widetilde{\mathcal{V}}^{H}_{[0,T]}$ the unique solution of the above equation.
Thanks to that $f^{\theta'}_{1}$ is increasing in $\theta'$, one has
\begin{equation*}
f_{1}^{\widetilde{Y}'_{1}}(t,x,y,z) \leq  f_{1}^{\widetilde{Y}'_{0}}(t,x,y,z), \qq \forall (t,x,y,z)\in [0,T]\times \mathbb{R}^{3}.
\end{equation*}
Therefore, similar to the above discussion we deduce
\begin{equation*}
   \widetilde{Y}_{2}(t)\leq \widetilde{Y}_{1}(t), \ \ a.s.
\end{equation*}
By induction, one can construct a sequence
$\{(\widetilde{Y}_{n}(\cdot),\widetilde{Z}_{n}(\cdot))\}_{n\geq 1} \subseteq \widetilde{\mathcal{V}}_{[0,T+K]} \times \widetilde{\mathcal{V}}^{H}_{[0,T]}$
such that
$$\left\{\ba{ll}
\ds \widetilde{Y}_{n}(t)=g_{1}(\eta_{T}) + \int_t^T \dbE'[f_{1}(s,\eta_{s},\widetilde{Y}_{n}(s),\widetilde{Z}_{n}(s),\widetilde{Y}'_{n-1}(s+\delta(s)))] ds -\int_t^T \widetilde{Z}_{n}(s) dB_{s}^{H}, \ \ \ \ t\in[0,T];\\
\ns\ds Y_{n}(t)=g_1(\eta_{t}),  \qq t\in[T,T+K].
\ea\right.$$
Similarly, we obtain
\begin{equation*}
  Y_{2}(t)= \widetilde{Y}_{0}(t)\geq \widetilde{Y}_{1}(t)\geq \widetilde{Y}_{2}(t)\geq
   \cdots \geq \widetilde{Y}_{n}(t)\geq \cdots,  \ \ a.s.
\end{equation*}
In the following, we show $\{(\widetilde{Y}_{n}(\cdot),\widetilde{Z}_{n}(\cdot))\}_{n\geq 1}$ is a Cauchy sequence. Denote
$$\hat{Y}_{n}=\widetilde{Y}_{n}-\widetilde{Y}_{n-1},\qq \hat{Z}_{n}=\widetilde{Z}_{n}-\widetilde{Z}_{n-1},
\qq n\geq 4.$$
From the estimate (\ref{34}), we have
$$\ba{ll}
\ds \mathbb{E}\left(\frac{\beta}{2}\int_0^T e^{\beta s}|\hat{Y}_{n}(s)|^{2} ds
+\frac{2}{M}\int_0^T s^{2H-1}e^{\beta s}|\hat{Z}_{n}(s)|^{2} ds\right)\\
\ns\ds \leq \frac{2}{\beta}\mathbb{E}\int_0^T e^{\beta s}
\bigg(\dbE'[f_{1}(s,\eta_{s},\widetilde{Y}_{n}(s),\widetilde{Z}_{n}(s),\widetilde{Y}'_{n-1}(s+\delta(s)))]\\
\ns\ds\qq\qq\qq \ \ \ -\dbE'[f_{1}(s,\eta_{s},\widetilde{Y}_{n-1}(s),\widetilde{Z}_{n-1}(s),\widetilde{Y}'_{n-2}(s+\delta(s)))] \bigg)^{2} ds. \qq
\ea$$
Then combining (H3) and Jensen's inequality,  note that $\delta$ satisfying (i) and (ii), one has

$$\ba{ll}
\ds \mathbb{E}\left(\frac{\beta}{2}\int_0^T e^{\beta s}|\hat{Y}_{n}(s)|^{2} ds
+\frac{2}{M}\int_0^T s^{2H-1}e^{\beta s}|\hat{Z}_{n}(s)|^{2} ds\right)\\
\ns\ds \leq \frac{6C}{\beta}\mathbb{E} \int_0^T e^{\beta s}\big(|\hat{Y}_{n}(s)|^{2}+s^{2H-1}|\hat{Z}_{n}(s)|^{2}\big) ds
     +\frac{6CL}{\beta}\mathbb{E} \int_0^T e^{\beta s}|\hat{Y}_{n-1}(s)|^{2} ds \\
\ns\ds \leq\frac{6C(L+1)}{\beta}\mathbb{E} \int_0^T e^{\beta s}\big(|\hat{Y}_{n}(s)|^{2}+s^{2H-1}|\hat{Z}_{n}(s)|^{2}+|\hat{Y}_{n-1}(s)|^{2}\big) ds.
\ea$$
Now we choose $M>2$ and let $\beta=12CM(L+1) + \frac{4}{M}$, then
$$\ba{ll}
\ds \mathbb{E}\int_0^T e^{\beta s}\big(|\hat{Y}_{n}(s)|^{2} + s^{2H-1}|\hat{Z}_{n}(s)|^{2}\big) ds\\
\ns\ds \leq \frac{1}{4}\mathbb{E} \int_0^T e^{\beta s}\big(|\hat{Y}_{n}(s)|^{2}+s^{2H-1}|\hat{Z}_{n}(s)|^{2}+|\hat{Y}_{n-1}(s)|^{2}\big) ds.
\ea$$
Hence
$$\ba{ll}
\ds \mathbb{E}\int_0^T e^{\beta s}\big(|\hat{Y}_{n}(s)|^{2} + s^{2H-1}|\hat{Z}_{n}(s)|^{2}\big) ds\\
\ns\ds\leq  \frac{1}{3}\mathbb{E}\int_0^T e^{\beta s}|\hat{Y}_{n-1}(s)|^{2}  ds\\
\ns\ds \leq  \frac{1}{3}\mathbb{E}\int_0^T e^{\beta s}\big(|\hat{Y}_{n-1}(s)|^{2} + s^{2H-1}|\hat{Z}_{n-1}(s)|^{2}\big) ds.\qq\ \ \ \ \
\ea$$
So
\begin{equation*}
      \mathbb{E}\int_0^T e^{\beta s}(|\hat{Y}_{n}(s)|^{2} + s^{2H-1}|\hat{Z}_{n}(s)|^{2}) ds
\leq  (\frac{1}{3})^{n-4}\mathbb{E} \int_0^T e^{\beta s}(|\hat{Y}_{4}(s)|^{2} + s^{2H-1}|\hat{Z}_{4}(s)|^{2}) ds.
\end{equation*}
It follows that $(\hat{Y}_{n}(\cdot))_{n\geq 4}$ and $(\hat{Z}_{n}(\cdot))_{n\geq 4}$ are respectively Cauchy sequences in
$\widetilde{\mathcal{V}}_{[0,T+K]}$ and  $\widetilde{\mathcal{V}}_{[0,T]}^{H}$.
Denote their limits by $\widetilde{Y}_{\cdot}$ and $\widetilde{Z}_{\cdot}$, respectively.
From Theorem \ref{th}, we have $\widetilde{Y}(t)= Y_{1}(t), \ \ a.s.$,
which deduce that
\begin{equation*}
 Y_{1}(t)\leq Y_{2}(t), \ \ a.s.
\end{equation*}
Therefore, the desired result is obtained.
\end{proof}

\begin{example}
 Suppose we are facing with the following two MF-ABSDEs,
$$\left\{\ba{ll}
\ds Y_{1}(t)=g_{1}(\eta_{T})+\int_t^T[Y_{1}(s)+Z_{1}(s)+\dbE' Y'_{1}(s+\delta(s))-1]ds-\int_t^TZ_{1}(s)dB_{s}^{H};\\
\ns\ds Y_{1}(t) = g_{1}(\eta_{t}),\qq t\in[T,T+K],
\ea\right.$$
and
$$\left\{\ba{ll}
\ds Y_{2}(t)=g_{2}(\eta_{T})+\int_t^T[Y_{2}(s)+Z_{2}(s)+\dbE' Y'_{2}(s+\delta(s))+1]ds-\int_t^TZ_{2}(s)dB_{s}^{H};\\
\ns\ds Y_{2}(t) = g_{2}(\eta_{t}),\qq t\in[T,T+K],
\ea\right.$$
where $g_{1}$ and $g_{2}$ satisfy (H1) with $g_{1}(x)\leq g_{2}(x), \ \forall x\in \mathbb{R}$.
Then, according to Theorem \ref{40}, one has
\begin{equation*}
 Y_{1}(t)\leq Y_{2}(t), \ \ a.s.
\end{equation*}
\end{example}

\section{Optimal control problem}

Let $\delta>0.$ We want to control a process $X(t) = X^{u}(t)$ given by an
equation of the form:
\bel{eq:1}\left\{\ba{ll}
\ds dX(t)=b(t, \mathbb{P}_{X(t)},\mathbb{P}_{X(t-\delta)},\mathbb{P}_{u(t)}) dt +\sigma(t)dB^{H}_{t},\qq t\in[0,T];\\
\ns\ds X(t)=x_{0}(t),\qq t\in[-\delta,0].
\ea\right.\ee

The function $\sigma$ is assumed to be in $\mathcal{H}$, the integral with respect to $B^{H}$ is therefore understood in the Wiener sense. The function $b : [0,T] \times \mathcal{P}_{2}(\mathbb{R}
) \times \mathcal{P}_{2}(\mathbb{R}
) \times \mathcal{P}_{2}(\mathbb{R}
) \rightarrow \mathbb{R}$ is assumed to be deterministic in the sense that it's a function of $t$ and the laws of processes $X$ and $u$. The function $x_{0}$ is assumed to be continuous and deterministic. The set $\mathcal{U}\subset%
\mathbb{R}
$ consists of the admissible control values. The information available to the
controller is given by the filtration $\mathbb{F}$ (generated by the fBm $B^{H}$). The set of admissible controls, i.e., the strategies available
to the controller, is given by  $\mathcal{A}_{\mathbb{F}}$ the set of  $\mathcal{U}$-valued and $\mathbb{F}$-adapted square integrable processes. In this paper, we assume that $X$ exists and belongs to $L^{2}(\Omega \times [0,T])$. For recent works about fractional stochastic differential equation, we refer the reader to
Ferrante and  Rovira \cite{FR2006},  Buckdahn et al. \cite{Buckdahn2017}, Buckdahn and Jing \cite{Buckdahn3}, etc. For other examples of stochastic optimal control problems with delay driven by fBm, the reader may consult Agram, Douissi and Hilbert \cite{DHA}.

\ms

The performance functional is assumed to have the following form:
\bel{P}J(u)=\dbE\bigg[g(X(T),\mathbb{P}_{X(T)})+\int_{0}^{T}f(t,X(t),X(t-\delta),\mathbb{P}_{X(t)},\mathbb{P}_{X(t-\delta)},u(t))dt\bigg],\ee
where $f:\Omega\times\left[
0,T\right]  \times%
\mathbb{R}
^{2}\times\mathcal{P}_{2}(\mathbb{R})^{2}\times\mathcal{U}\rightarrow%
\mathbb{R}
$ and $g:\Omega  \times%
\mathbb{R}
\times\mathcal{P}_{2}(\mathbb{R})\rightarrow%
\mathbb{R}
$ are given processes, such that for all $ t \in [0,T]$, $f(.,t,x,\bar{x},m,\bar{m},u)$ is assumed to be $\mathcal{F}_{t}$-measurable for all $x,\bar{x} \in \mathbb{R}$, $m,\bar{m} \in \mathcal{P}_{2}(\mathbb{R})$,
$u\in\cU$. The process $g(.,x,m)$ is assumed to be $\mathcal{F}_{T}$-measurable for all $x \in \mathbb{R}$, $m \in \mathcal{P}_{2}(\mathbb{R})$.

\ms

We also assume the following integrability condition
\bel{PP}\dbE\bigg[\Big|g(X(T),\mathbb{P}_{X(T)})\Big|
+\int_{0}^{T}\Big|f(t,X(t),X(t-\delta),\mathbb{P}_{X(t)},\mathbb{P}_{X(t-\delta)},u(t))\Big|dt\bigg] < +\infty.\ee
The functions $b$, $f$ and $g$ are assumed to be continuously differentiable w.r.t $x, \bar{x},u$ with bounded
derivatives and admit Fr\'{e}chet bounded derivatives with respect to the probability measures as mentioned in the preliminaries.

\ms

The problem we consider in this section is the following:

\ms

\textbf{Problem:} Find a control $u^{*} \in \mathcal{A}_{\mathbb{F}}$ such that
\begin{equation}
\label{eq:optimal perf}
J(u^*)=\sup_{u\in \mathcal{A}_{\mathbb{F}}}J(u).
\end{equation}
Any control $u^{*} \in \mathcal{A}_{\mathbb{F}}$ satisfying (\ref{eq:optimal perf}) is called an optimal control.

\ms

The Hamiltonian associated to our problem is defined by%
$$
H:\Omega\times\left[  0,T\right]  \times%
\mathbb{R}
\times
\mathbb{R}\times \mathcal{U}\times \mathcal{P}_{2}(\mathbb{R})\times \mathcal{P}_{2}(\mathbb{R})\times \mathcal{P}_{2}(\mathbb{R})\times%
\mathbb{R}
\times%
\mathbb{R}
\rightarrow%
\mathbb{R}
$$
with%
\begin{equation}%
\begin{array}
[c]{ll}%
H(t,x,\overline{x},u,m_{1},m_{2},m_{3},y,z)  = f(t,x,\overline
{x},m_{1},{m}_{2},u)+y \times b(t,m_{1},m_{2},m_{3}) +%
z \times \sigma(t).
\end{array}
\label{ham}%
\end{equation}

For $u\in\mathcal{A}_{\mathbb{F}}$ with corresponding solution $X=X^{u}$,
define, whenever solutions exist, $(Y,Z) \triangleq (Y^{u},Z^{u})$, by the adjoint equation, in terms of the Hamiltonian, as follows:
\bel{eq:2}\left\{\ba{ll}
\ds dY(t) = -\{ \partial_{x}H(t)+ {\dbE}[\partial_{\bar{x}}H(t+\delta)\chi_{[0,T-\delta]}(t)|\mathcal{F}_{t}] + {{\dbE}^{\prime}}[\partial_{m_{1}}{H}^{\prime}(t)(X(t))] \\
\ns\ds\qq\qq +  {{\dbE}[{\dbE}^{\prime}}[\partial_{{m}_{2}}{H}^{\prime}(t+\delta)(X(t)) \chi_{[0,T-\delta]}(t)]|\mathcal{F}_{t}]\}dt + Z(t)dB^{H}%
(t),\qq t\in[0,T],\\
\ns\ds Y(T)=\partial_{x}g(T)+{{\dbE}^{\prime}}[\partial_{m} {g}^{\prime}(T)({X}(T))].
\ea\right.\ee
Note that we have used the following notations:
$$\ba{ll}
\ds H(t)\triangleq H(t,X(t),X(t-\delta),u(t), \mathbb{P}_{X(t)}, \mathbb{P}_{X(t-\delta)},\mathbb{P}_{u(t)},Y(t),Z(t)), \\
\ns\ds {H}^{\prime}(t)\triangleq H(t, {X}^{\prime}(t), {X}^{\prime}(t-\delta), {u}^{\prime}(t), \mathbb{P}_{X(t)}, \mathbb{P}_{X(t-\delta)}, \mathbb{P}_{u(t)} , {Y}^{\prime}(t), {Z}^{\prime}(t)),\\
\ns\ds g(T)\triangleq g(X(T),\mathbb{P}_{X(T)}), \qq {g}^{\prime}(T)\triangleq g({X}^{\prime}(T),\mathbb{P}_{X(T)}).
\ea$$
\begin{remark}
\begin{enumerate}
	\item Note that according to the definition of the differentiability of functions of measures and Remark \ref{Rem}, the terminal value $Y(T)$ is a measurable function of $X(T)$ and the value of $Z(T)$ follows from the Clark-Ocone formula, see \cite{BOZ}.
	\item In the coming example, we illustrate how to solve a special kind of BSDE (\ref{eq:2}), the resolution proposed is done on time intervals using the results we obtained in the previous sections concerning the existence and uniqueness of MF-ABSDE (\ref{31}) when the constant $K$ is equal to zero.

\end{enumerate}
\end{remark}
\subsection{Sufficient maximum principle}

In this section, we prove sufficient stochastic maximum principle.
\begin{theorem}\label{Suff} Let $u^{\ast}%
\in\mathcal{A}_{\mathbb{F}}$, with corresponding controlled state process $X_{*}\triangleq X^{u^{*}}$. Suppose
that there exists $(Y_{*}(t)$, $Z_{*}(t))$, the solution
of the associated adjoint equation $\left(\ref{eq:2}\right)$. Assume the following:

\begin{enumerate}
 \item $(X^{u}(t)Z_{*}(t)) \in dom(\delta^{H})$ \text{ \ } $\forall
  u \in \mathcal{A}_{\mathbb{F}}$.
\item (Concavity) The functions
$$
\begin{array}
[c]{ll}%
(x,\bar{x},u,m_{1},m_{2},m_{3}) & \mapsto H(t,x,\bar{x},u,m_{1},m_{2},m_{3},Y_{*}(t),Z_{*}(t)) ,\\
(x,m) & \mapsto g(x,m)\text{,}%
\end{array}
$$
are concave  for each $t \in [0,T]$ almost surely.

\ms

Moreover, the control $u^{*}$ satisfies the following conditions:
\item (Maximum condition)
\begin{center}\label{maxQ1}
$H(t, X_{*}(t),X_{*}(t-\delta),u^{*}(t),\mathbb{P}_{X_{*}(t)}, \mathbb{P}_{X_{*}(t-\delta)}, \mathbb{P}_{u^{*}(t)},Y_{*}(t),Z_{*}(t))
= \hspace{2cm}\underset{u\in \mathcal{U}}{\text{ }\sup} { \ } H(t, X_{*}(t),X_{*}(t-\delta),u, \mathbb{P}_{X_{*}(t)},\mathbb{P}_{X_{*}(t-\delta)},\mathbb{P}_{u^{*}(t)}, Y_{*}(t),Z_{*}(t))$,
\end{center}
for all $t \in [0,T]$ almost surely.
\item $\partial_{m_{3}} H(t,X_{*}(t),X_{*}(t-\delta),u^{*}(t), \mathbb{P}_{X_{*}(t)}, \mathbb{P}_{X_{*}(t-\delta)}, \mathbb{P}_{u^{*}(t)},Y_{*}(t),Z_{*}(t))(u^{*}(t)) = 0$,\\
\\
for all $t \in [0,T]$ almost surely.
\end{enumerate}

Then $(u^{\ast},X_{\ast})$ is an optimal couple for our problem.
\end{theorem}
\begin{remark}
The above condition 4 means that the Fr\'{e}chet derivative of $H$ with respect to the law of the control in $u^{*}$ vanishes.
\end{remark}
\begin{proof}
Let $u\in\mathcal{A}_{\mathbb{F}}$ be a generic admissible control. By the definition of the performance functional $J$ given by
$\left(  \ref{P}\right)  $, we have
\begin{equation}%
\begin{array}
[c]{lll}%
J(u)-J(u^{\ast}) & = & A_{2} + A_{3},
\end{array}
\label{J2}%
\end{equation}
where
$$\ba{ll}
\ds A_2\triangleq \mathbb{E}\Big[\int_{0}^{T}[f(t)-f_{\ast}(t)]dt\Big],\\
\ns\ds A_3\triangleq \mathbb{E}\Big[g(T)-g_{\ast}(T)\Big].
\ea$$
Applying the definition of Hamiltonian $\left(  \ref{ham}\right)  $, we
have%
\bel{j1}A_{2}=\dbE\Big[\int_{0}^{T}\Big(H(t)-H_{\ast}(t)-Y_{*}(t)\bar{b}(t) \Big) dt\Big] ,\ee
where we used the following notations
$$%
\begin{array}
[c]{ll}%
b(t) & \triangleq b(t, \mathbb{P}_{X(t)},\mathbb{P}_{X(t-\delta)},\mathbb{P}_{u(t)}),\qq b_{*}(t)  \triangleq b(t,{P}_{X_{*}(t)},\mathbb{P}_{X_{*}(t-\delta)}, \mathbb{P}_{u^{*}(t)}),\\
f(t) &\triangleq f(t,X(t),X(t-\delta),\mathbb{P}_{X(t)},\mathbb{P}_{X(t-\delta)},u(t)), \\
f_{*}(t) & \triangleq f(t,X_{*}(t),X_{*}(t-\delta),\mathbb{P}_{X_{*}(t)},\mathbb{P}_{X_{*}(t-\delta)},u^{*}(t)), \\
g(T) & \triangleq g(X(T), \mathbb{P}_{X(T)}), \qq  g_{*}(T)  \triangleq g(X_{*}(T),\mathbb{P}_{X_{*}(T)}),\\
H(t) & \triangleq H(t,X(t),X(t-\delta),u(t),\mathbb{P}_{X(t)},\mathbb{P}_{X(t-\delta)}, \mathbb{P}_{u(t)},Y_{*}(t),Z_{*}(t)), \\
H_{*}(t) & \triangleq H(t,X_{*}(t),X_{*}(t-\delta),u^{*}(t),\mathbb{P}_{X_{*}(t)},\mathbb{P}_{X_{*}(t-\delta)},\mathbb{P}_{u^{*}(t)},Y_{*}(t),Z_{*}(t)), \\
\bar{b}(t) & \triangleq  b(t)-b_{\ast}(t),\qq
\bar{X}(t)  \triangleq X(t)- X_{\ast}(t).
\end{array}
$$
\begin{remark}\label{rem}
Notice that since $\sigma$ is a function of $t$ and therefore it is not related to the process $X$, we have $d\bar{X}({t}) = \bar{b}({t})dt$.
\end{remark}
Now using the concavity of $g$ and the terminal value of BSDE (\ref{eq:2}) associated to  $(u^{*},X_{*})$, we get by Fubini's theorem
$$\ba{ll}
\ds A_{3} \leq {\dbE}[\partial_{x}g_{*}(T)\bar{X}(T)]  + {\dbE}[{{\dbE}^{\prime}}[\partial_{m}g_{*}(T)({X}^{\prime}_{\ast}(T))\bar{X}^{\prime}(T)]]\\
\ns\ds\ \ \ \ = {\dbE}[(\partial_{x}g_{*}(T)
+ {{\dbE}^{\prime}}[\partial_{m}{g}^{\prime}_{*}(T)({X}_{\ast}(T))])\bar{X}(T)] \\
\ns\ds \ \ \ \ = {\dbE}[Y_{*}(T)\bar{X}(T)].
\ea$$
Applying the integration by parts formula (Proposition \ref{3}) to $\bar{X}(t)$ and $Y_{*}(t)$, we get
\begin{equation*}
d(Y_{*}(t)\bar{X}(t))  = Y_{*}(t)d\bar{X}(t) + \bar{X}(t)dY_{*}(t).
\end{equation*}
The equality comes from Remark \ref{rem} and the fact that  $\mathbb{D}^{H}_{t}\bar{X}({t})=0$ because $\bar{X}(t) =  \int_{0}^{t} \bar{b}(s) ds$ and $\bar{b}$ is deterministic.
Hence, integrating from $0$ to $T$, taking the expectation and using the first assumption, we get
\bel{PXprim}\ba{ll}
\ds {\dbE}[Y_{*}(T)\bar{X}(T)] = {\dbE}[\int_{0}^{T} Y_{*}(t)d\bar{X}(t)]
+ {\dbE}[\int_{0}^{T}\bar{X}(t)dY_{*}(t)]  \\
\ns\ds = {\dbE}[\int_{0}^{T}Y_{*}(t)\bar{b}(t)dt] - {\dbE}[\int_{0}^{T} \bar{X}(t)\{\partial_{x}{H}_{*}(t)  +
 \partial_{\bar{x}}H_{*}(t+\delta) \chi_{[0,T-\delta]}(t) \\
\ns\ds \ \ \ +  {{\dbE}^{\prime}}[\partial_{{m}_{1}}{H}^{\prime}_{*}(t)({X}_{*}(t))] +  {\dbE}^{\prime}[\partial_{{m}_{2}} {H}^{\prime}_{*}(t+\delta)({X}_{*}(t))]\chi_{[0,T-\delta]}(t)\}dt]  \\
\ns\ds = {\dbE}[\int_{0}^{T}Y_{*}(t)\bar{b}(t)dt]- {\dbE}[\int_{0}^{T} \bar{X}(t)\partial_{x}H_{*}(t) dt] -
{\dbE}[\int_{0}^{T} \partial_{\bar{x}}H_{*}(t) \bar{X}(t-\delta) dt] \\
\ns\ds \ \ \ - {\dbE}[\int_{0}^{T} {{\dbE}^{\prime}}[\partial_{m_{1}}H_{*}(t)({X}^{\prime}_{*}(t))\bar{X}^{\prime}(t)]dt]- {\dbE}[\int_{0}^{T} {{\dbE}^{\prime}}[\partial_{{m}_{2}}H_{*}(t)({X}^{\prime}_{*}(t-\delta))\bar{X}^{\prime}(t-\delta)]dt].
\ea\ee
To obtain the last equality, we did the following change of variables  $r = t + \delta$ to get
\begin{equation*}
{\dbE}[\int_{0}^{T-\delta} \bar{X}(t) \partial_{\bar{x}}H_{*}(t+\delta)dt]  = {\dbE}[\int_{\delta}^{T} \bar{X}(r-\delta) \partial_{\bar{x}}H_{*}(r)dr]
= {\dbE}[\int_{0}^{T} \bar{X}(r-\delta) \partial_{\bar{x}}H_{*}(r)dr],
\end{equation*}
where we used that ${\dbE}[\int_{0}^{\delta} \bar{X}(r-\delta) \partial_{\bar{x}}H_{*}(r)dr] =  {\dbE}[\int_{- \delta}^{0} \bar{X}(u) \partial_{\bar{x}}H_{*}(u+\delta)du] = 0$, since $\bar{X}(u) = 0$ for all $u \in [-\delta,0]$, because $X_{*}(t) = X(t) = x_{0}(t)$ for all $t  \in [-\delta,0]$. \

\ms

Similarly, we get using the previous argument and by Fubini's theorem
\begin{equation*}
{\dbE}[\int_{0}^{T} \bar{X}(t) {{\dbE}^{\prime}}[\partial_{{m}_{2}} {H}^{\prime}_{*}(t+\delta)(X_{*}(t))]\chi_{[0,T-\delta]}(t)dt] = {\dbE}[\int_{0}^{T} {{\dbE}^{\prime}}[\partial_{{m}_{2}}{H}_{*}(t)({X}^{\prime}_{*}(t-\delta))\bar{X}^{\prime}(t-\delta)]dt].
\end{equation*}
By (\ref{J2}), (\ref{j1}) and (\ref{PXprim}), we obtain
$$\ba{ll}
\ds J(u)-J(u^{*}) \leq {\dbE}[\int_{0}^{T} (H(t)-H_{*}(t))dt] - {\dbE}[\int_{0}^{T} \partial_{x}H_{*}(t) \bar{X}(t)dt] - {\dbE}[\int_{0}^{T} \partial_{\bar{x}}H_{*}(t) \bar{X}(t-\delta)dt]\\
\ns\ds\qq\qq\qq \ \ \ -
{\dbE}[\int_{0}^{T}
{{\dbE}^{\prime}}[\partial_{{m}_{1}}H_{*}(t)({X}^{\prime}_{*}(t))\bar{X}^{\prime}(t)]dt]-
{\dbE}\int_{0}^{T} {{\dbE}^{\prime}}[\partial_{{m}_{2}}H_{*}(t)({X}^{\prime}_{*}(t-\delta))\bar{X}^{\prime}(t-\delta)]dt \\
\ns\ds\qq\qq\qq \leq 0.
\ea$$
Due to the concavity assumption on $H$ and because $u^{*}$ satisfies the maximum condition 3 and 4. the first order derivative of $H$ in $u^{*}$ and the Fr\'echet derivative of $H$ with respect to the law of the control $u^{*}$ in $u^{*}(t)$ vanish.
\end{proof}

\subsection{Application and example}

The main applications of mean-field dynamics that appear in the literature rely mainly on a dependence upon the probability measures through functions of scalar moments of the measures. More precisely, we assume that:
$$\ba{ll}
\ns\ds\ \ \ \ b(t,m_{1},m_{2},m_{3}) = \hat{b}(t,(\psi_{1},m_{1}),(\psi_{2},{m}_{2}),(\psi_{3},{m}_{3})),\\
\ns\ds f(t,x,\bar{x},m_{1},{m}_{2},u) = \hat{f}(t,x,\bar{x},(\gamma_{1},m_{1}),(\gamma_{2},{m}_{2}),u),\\
\ns\ds\qq\qq \ \ \ g(x,m) = \hat{g}(x,(\gamma_{3},m)).
\ea$$
for some scalar differentiable functions $\psi_{1}$, $\psi_{2}$, $\psi_{3}$, $\gamma_{1}$, $\gamma_{2}$, $\gamma_{3}$ with at most quadratic growth at $\infty$. The function $\hat{b}$ is defined on $[0,T] \times \mathbb{R}\times \mathbb{R} \times \mathbb{R}$, the function $\hat{f}$ is defined on $[0,T] \times \mathbb{R} \times \mathbb{R} \times \mathbb{R} \times \mathbb{R} \times \mathcal{U}$ and $\hat{g}$ is defined on $\mathbb{R} \times \mathbb{R}$. The notation $(\psi, m)$ denotes the integral of the function $\psi$ with respect to the probability measure $m$. The Hamiltonian that we defined in the previous section takes now the following form:
$$\ba{ll}
\ds H(t,x,\overline{x},u,m_{1},m_{2},m_{3},y,z) \\
\ns\ds= \hat{f}(t,x,\overline
{x},(\gamma_{1},m_{1}),(\gamma_{2},{m}_{2}),u)+ y \times \hat{b}(t,(\psi_{1},m_{1}),(\psi_{2},{m}_{2}),(\psi_{3},{m}_{3})) + z \times \sigma(t).
\ea$$
The functions $\hat{f}$, $\hat{b}$ and $\hat{g}$ are similar to the functions $f$, $b$, $g$, the only difference is that the measure for example $m_{1}$ is replaced by a numeric variable say $x^{\prime}$. Therefore according to the definition of the differentiability with respect to functions of measures recalled in the preliminaries, the derivative of the Hamiltonian with respect to the measure $m_{1}$ for instance, is computed as follows,
$$\ba{ll}
\ds \partial_{m_{1}}H(t,x,\overline{x},u,m_{1},m_{2},m_{3},y,z)(x^{\prime}) \\
\ns\ds= \partial_{x^{\prime}} \hat{f}(t,x,\overline
{x},(\gamma_{1},m_{1}),(\gamma_{2},{m}_{2}),u)  \gamma_{1}^{\prime}(x^{\prime})  + y \times \partial_{x^{\prime}}\hat{b}(t,(\psi_{1},m_{1}),(\psi_{2},{m}_{2}),(\psi_{3},{m}_{3})) \psi_{1}^{\prime}(x^{\prime}).
\ea$$
The terminal value of the adjoint BSDE (\ref{eq:2}) which is $Y(T)  = \partial_{x}g(T)+ {\dbE}^{\prime}[\partial_{m}{g}^{\prime}(T)({X}(T))]$, can be written in terms of the derivatives of the function $\hat{g}$ as follows:
\begin{equation*}
Y(T)  = \partial_{x}\hat{g}(X_{T},{\dbE}[\gamma_{3}(X_{T})])+ {\dbE}^{\prime}[\partial_{x^{\prime}}\hat{g}({X}^{\prime}_{T},{\dbE}[\gamma_{3}(X_{T})])]\gamma_{3}^{\prime}(X_{T}).
\end{equation*}

\subsubsection{Example}

We consider now a controlled state process $X = X^{\alpha}$ given by the following mean-field delayed stochastic differential equation:
\bel{wealth1}\left\{\ba{ll}
 dX(t) =  -[\beta_{1}(t)\dbE[X(t-\delta)]+ \beta_{2}\dbE[\alpha(t)]^{2}]dt
+ \beta_{3}(t)dB^{H}(t) , \qq t \in [0,T];\\
\ns\ds X(t) = x_{0}(t),\qq t \in [-\delta,0],
\ea\right.\ee
where $\delta > 0$ is a given constant, $\beta_{1}$, $x_{0}$ are given bounded deterministic functions, $\beta_{2}$ is a given positive constant, $\beta_{3}$ is a given deterministic function in $\mathcal{H}$. The integral with respect to the fBm is therefore a Wiener type integral and $\alpha \in \mathcal{A}_{\mathbb{F}}$ is our control process. The set $\mathcal{A}_{\mathbb{F}}$ are the admissible controls assumed to be square integrable $\mathbb{F}$-adapted processes with real positive values.
\\
\\
We want to minimize the expected value of $X_{T}^{2}$  with a minimal average use of energy, measured by the integral $\int_{0}^{T}{\dbE}[\alpha^{2}(t)]dt$, more precisely, the performance functional we consider in this example has the following form:
\begin{equation}\label{LQJ}
J(\alpha) = - \frac{1}{2} \Big(\dbE[ X_{T}^{2}] + \dbE[ \int_{0}^{T}\alpha^{2}(t)dt]\Big).
\end{equation}
Our goal is therefore to find the control process $\alpha^{*} \in \mathcal{A}_{\mathbb{F}}$, such that
\begin{equation}\label{LQ}
J(\alpha^{*}) = \sup_{\alpha \in \mathcal{A}_{\mathbb{F}}} J(\alpha).
\end{equation}
The Hamiltonian of our control problem is the following,
\begin{equation*}
H(t,x,\bar{x},\alpha,m_{1},m_{2},m_{3},y,z) = -\frac{1}{2} \alpha^{2} -y[\beta_{1}(t)(\textsc{Id},m_{2}) + \beta_{2}(\textsc{Id},m_{3})^{2}] + \beta_{3}(t) z.
\end{equation*}
So, according to the notations we used previously, we have\\
\\
$* \partial_{x}H(t) = 0$, \\
$* \partial_{\bar{x}}H(t) = 0$, \\
$* \partial_{m_{1}}H(t)(X(t))=0$, \\
$* \partial_{m_{2}}H(t)(X(t-\delta)) = - y \beta_{1}(t)$,\\
$* \partial_{m_{3}}H(t)(\alpha(t)) = -2  y \beta_{2} \mathbb{E}[\alpha(t)]$, \\
$* \partial_{\alpha}H(t) = - \alpha$.\\
\\
Hence, by calculating the second derivatives of $H$, we find that the Hessian matrix is semi definite negative and therefore the Hamiltonian $H$ is concave in $(x,\bar{x},\alpha,m_{1},m_{2},m_{3})$ under the condition $y \geq 0$.

\ms

Moreover the function $\alpha \in \mathbb{R}_{+} \mapsto H(t,x,\bar{x},\alpha,m_{1},m_{2},m_{3},y,z)$ is  concave and decreasing and therefore is maximal in $\alpha^{*} \triangleq 0$, note that once evaluating the derivative of $H$ with respect to $m_{3}$ in $\alpha^{*}$, we get $\partial_{m_{3}}H(t)(\alpha^{*}(t)) = 0$.

\ms

On the other hand the adjoint solution of the BSDE of our dynamic satisfies the following BSDE:
\begin{equation}\label{226}
\left\{
\begin{array}
[c]{lll}%
dY(t) =   \beta_{1}(t+\delta){\dbE}^{\prime}[Y^{\prime}(t+\delta)\chi_{[0,T-\delta]}(t)] dt + Z(t) dB^{H}(t), \qq t \in [0,T],\\
Y(T)  = -X_{T}.
\end{array}
\right.  %
\end{equation}
We propose a resolution of the previous anticipated BSDE by solving a sequence of linear BSDEs following this procedure:
\\
\\
\textbf{Step 1. } If $t \in [T-\delta,T]$, the previous BSDE takes the form
\begin{equation*}
\left\{
\begin{array}
[c]{lll}%
dY(t)  =  Z(t) dB^{H}(t) , \text{ \ \ } t \in  [T-\delta,T];\\
Y(T)  =  -X_{T}.
\end{array}
\right.  %
\end{equation*}
Then, under the hypothesis of Theorem \ref{FirstMethod}, this BSDE has a unique solution $(Y,Z)$ in $\tilde{\mathcal{V}}_{[T-\delta,T]} \times \tilde{\mathcal{V}}^{H}_{[T-\delta,T]}$.
\\
\\
\textbf{Step 2.} If $t \in [T-2\delta, T-\delta]$ and $T- 2 \delta > 0$, we obtain the BSDE
\begin{equation*}
\left\{
\begin{array}
[c]{lll}%
dY(t)  =  \beta_{1}(t+\delta){\dbE}^{\prime}[Y^{\prime}(t+\delta)] + Z(t) dB^{H}(t) , \qq t \in  [T- 2\delta,T - \delta];\\
Y(T- \delta)= \text{ \ } \text{known from step 1}.
\end{array}
\right.  %
\end{equation*}
We set  $ \psi_{\delta}(t)\triangleq \beta_{1}(t+\delta){\dbE}^{\prime}[Y^{\prime}(t+\delta)]$ which is the driver of this BSDE,  as $\psi_{\delta}(.)$ checks the hypothesis of Theorem \ref{FirstMethod}, this BSDE has a unique solution $(Y,Z)$ in $\tilde{\mathcal{V}}_{[T- 2\delta,T - \delta]} \times \tilde{\mathcal{V}}^{H}_{[T- 2\delta,T - \delta]}$.\\
\\
We continue like this by induction up to and including step n, where n is such that $T - n\delta \leq 0 < T- (n-1)\delta$ and we solve the corresponding BSDE on the time interval $[0, T-(n-1)\delta]$ and we solve the corresponding BSDE on the time interval $[0, T-(n-1)\delta]$.

\ms

According to Theorem \ref{Suff} and the previous calculus, an optimal decision of our control problem is the constant control $\alpha^{*} = 0$, the value of the performance functional in $\alpha^{*}$ is $J(\alpha^{*}) = - \frac{1}{2} \dbE[X_{T}^{2}]$, where $X \triangleq X^{\alpha^{*}}$ is the solution of the SDE (\ref{wealth1}), thus we have the following corollary.

\begin{corollary}
	The constant control $\alpha^{*} = 0$ is an optimal control for the control problem (\ref{LQ}), the corresponding triplet $(X^{\alpha^{*}},Y^{\alpha^{*}},Z^{\alpha^{*}})$ solves the couple of systems (\ref{wealth1}) and (\ref{226}) of (decoupled) forward-backward stochastic differential equations, and the value of the performance functional in the proposed optimal control $\alpha^{*}$ is $J(\alpha^{*}) = - \frac{1}{2} \dbE[X_{T}^{2}]$, where $X_{T} \triangleq X_{T}^{\alpha^{*}}$.
\end{corollary}

\bibliographystyle{elsarticle-num}

\begin{thebibliography}{99}

\bibitem{Bender}
   C. Bender,
   Backward SDEs driven by Gaussian processes,
   Stochastic Process. Appl. 124 (2014) 2892-2916.


\bibitem{BOZ} F. Biagini, Y. Hu, B. {\O}ksendal, T. Zhang,
 Stochastic Calculus for Fractional Brownian motion and Applications,
 Springer, London, 2008.

\bibitem{Borkowska} K.J. Borkowska,
    Generalized BSDEs driven by fractional Brownian motion,
    Statist. Probab. Lett. 83 (2013) 805-811.

\bibitem{Buckdahn}
   R. Buckdahn, B. Djehiche, J. Li, S. Peng,
   Mean-field backward stochastic differential equations: A limit approach,
   Ann. Probab. 37 (2009) 1524-1565.

\bibitem{Buckdahn2}
   R. Buckdahn, J. Li, S. Peng,
   Mean-field backward stochastic differential equations and related partial differential equations,
   Stochastic Process. Appl. 119 (2009) 3133-3154.

\bibitem{Buckdahn2017}
   R. Buckdahn, J. Li, S. Peng, C. Rainer,
   Mean-field stochastic differential equations and associated PDEs,
   Ann. Probab. 45(2) (2017) 824-878.

\bibitem{Buckdahn3}
   R. Buckdahn, S. Jing,
   Mean-field SDE driven by a fractional Brownian motion and related stochastic control problem,
   SIAM J. Control Optim. 55(3) (2017) 1500-1533.

\bibitem {carmona2} R. Carmona, F. Delarue,
  Forward-backward stochastic differential equations and controlled Mckean-Vlasov dynamics.
  Ann. Probab.  43 (2015) 2647-2700.

\bibitem {Cardaliaguet} P. Cardaliaguet,
  Notes on mean field games (from P. L. Lions' lectures at Coll\`{e}ge de France).
   Available at https://www.ceremade.dauphine.fr/~cardalia/  (2013).

\bibitem{DHA} S. Douissi, A. Hilbert, N. Agram, Mean-field optimal control problem of SDDEs driven by fractional Brownian motion. https://arxiv.org/pdf/1706.06233.pdf  (2018).

\bibitem{Decreusefond}
   L. Decreusefond, A.S. \"{U}st\"{u}nel,
   Stochastic analysis of the fractional Brownian motion,
   Potential Anal. 10 (1999) 177-214.


\bibitem{Peng2}
   N. El Karoui, S. Peng, M.C. Quenez,
   Backward stochastic differential equations in finance,
   Math. Finance 7 (1997) 1-71.

\bibitem{FR2006} M. Ferrante, C. Rovira,
   Stochastic delay differential equations driven by fractional Brownian motion with Hurst parameter $H>1/2$.
   Bernoulli 12(1) (2006) 85-100.

\bibitem{Hu3}
   Y. Hu,
   Integral transformations and anticipative calculus for fractional Brownian motions,
   Mem. Amer. Math. Soc. 175 (2005) no. 825.


\bibitem{Hu1}
   Y. Hu, D. Ocone, J. Song,
   Some results on backward stochastic differential equations driven by fractional Brownian motions,
   Stoch. Anal. Appl. Finance (2012) 225-242.

\bibitem{Hu} Y. Hu, S. Peng,
   Backward stochastic differential equation driven by fractional Brownian motion,
   SIAM J. Control Optim. 48 (2009) 1675-1700.

\bibitem{Lasry}
   J.M. Lasry, P.L. Lions,
   Mean field games,
   Japan. J. Math. 2 (2007) 229-260.

\bibitem{Maticiuc}
   L. Maticiuc, T. Nie,
   Fractional backward stochastic differential equations and fractional backward variational inequalities,
   J. Theory Probab. 28 (2015) 337-395.

\bibitem{Nualart}
   D. Nualart,
   The Malliavin Calculus and Related Topics (Second Edition),
   Springer, 2006.

\bibitem{Peng}
   E. Pardoux, S. Peng,
   Adapted solution of a backward stochastic differential equation,
   Systems Control Lett. 4 (1990) 55-61.

\bibitem{Peng92}
   E. Pardoux, S. Peng,
   Backward SDEs and quasi-linear PDEs,
   Lecture Notes in Control and Inform Sci. 176 (1992) 200-217.

\bibitem{PY2009}S. Peng, S., Z. Yang,
   Anticipated backward stochastic differential equations,
   Ann. Probab. 37 (2009) 877-902.

\bibitem{Wen} J. Wen, Y. Shi,
  Anticipative backward stochastic differential equations driven by fractional Brownian motion,
  Statist. Probab. Lett.  122 (2017) 118-127.

\bibitem{WS} J. Wen, Y. Shi,
  Mean-field backward stochastic differential equations driven by fractional Brownian motion, ArXiv.org/pdf/1606.02014 (2017).

\bibitem{Yong5}
J. Yong, X. Zhou,
Stochastic Controls: Hamiltonian Systems and HJB Equations,
Springer-Verlag, New York, 1999.

\end{thebibliography}

\end{document}